\documentclass[11pt,a4paper]{article}
\usepackage[utf8]{inputenc}
\usepackage[T1]{fontenc}
\usepackage{amsmath}
\usepackage{amsthm}
\usepackage{amsfonts}
\usepackage{amssymb}
\usepackage{mathrsfs}
\usepackage{url}
\usepackage{mathdots}
\usepackage{geometry}
\usepackage{verbatim}
\usepackage{hyperref}

\hypersetup{
    bookmarks=true,
    unicode=false,
    pdftoolbar=true,
    pdfmenubar=true,
    pdffitwindow=false,
    pdfstartview={FitH},
    pdftitle={Arthur's multiplicity formula for certain inner forms of special orthogonal and symplectic groups},
    pdfauthor={Olivier Taïbi},
    colorlinks=true,
    linkcolor=blue,
    citecolor=green,
    filecolor=green,
    urlcolor=cyan}

\newcommand{\Q}{\mathbb{Q}}
\newcommand{\A}{\mathbb{A}}
\providecommand{\C}{\mathbb{C}}
\renewcommand{\C}{\mathbb{C}}

\newcommand{\Z}{\mathbb{Z}}

\newcommand{\R}{\mathbb{R}}

\newcommand{\Gal}{\mathrm{Gal}}

\newtheorem{theo}{Theorem}[subsection]

\newtheorem{prop}[theo]{Proposition}
\newtheorem{rema}[theo]{Remark}
\newtheorem{exam}[theo]{Example}

\begin{document}

\baselineskip=16pt
\title{Arthur's multiplicity formula for certain inner forms of special orthogonal and symplectic groups}
\author{Olivier Taïbi \thanks{Imperial College London, United Kingdom. The author is supported by ERC Starting Grant 306326.}}

\date{}

\maketitle

\begin{abstract}
Let $\mathbf{G}$ be a special orthogonal group or an inner form of a symplectic group over a number field $F$ such that there exists a non-empty set $S$ of real places of $F$ at which $\mathbf{G}$ has discrete series and outside of which $\mathbf{G}$ is quasi-split.
We prove Arthur's multiplicity formula for automorphic representations of $\mathbf{G}$ having algebraic regular infinitesimal character at all places in $S$.
\end{abstract}

\tableofcontents
\newpage

\section{Introduction}

Let $F$ be a number field, and consider $\mathbf{G}$ a special orthogonal group or an inner form of a symplectic group over $F$.
Assume that there exists a \emph{non-empty} set $S$ of real places of $F$ such that
\begin{itemize}
\item for any place $v \in S$, the group $\mathbf{G}(F_v)$ admits discrete series representations,
\item for any place $v \not\in S$, the reductive group $\mathbf{G}_{F_v}$ is quasi-split.
\end{itemize}
This paper is devoted to the proof of Arthur's multiplicity formula for the subspace
$$L^2_{\mathrm{disc}}(\mathbf{G}(F) \backslash \mathbf{G}(\A_F))^{S-\mathrm{alg.\ reg.}}$$
of the discrete automorphic spectrum $L^2_{\mathrm{disc}}(\mathbf{G}(F) \backslash \mathbf{G}(\A_F)) \subset L^2(\mathbf{G}(F) \backslash \mathbf{G}(\A_F))$ corresponding to algebraic regular representations (``C-algebraic regular'' in the sense of \cite{BuzGee}) at all places in $S$, that is Theorem \ref{theo:main}, using the stabilisation of the trace formula for $\mathbf{G}$.
As in the case of quasi-split groups \cite[Theorem 1.5.2]{Arthur}, in the even orthogonal case we obtain a multiplicity formula only ``up to outer automorphism at all places of $F$''.

In particular we obtain the multiplicity formula for the full automorphic spectrum under the stronger assumption that $\mathbf{G}(\R \otimes_{\Q} F)$ is compact and $G$ is quasi-split at all finite places $F$.
The interest of this case is that automorphic forms for such a group $\mathbf{G}$ are more concrete than for quasi-split groups.
This allows one to compute explicitly and relatively easily with spaces of such automorphic forms.
Moreover, this assumption on $\mathbf{G}(\R \otimes_{\Q} F)$ allows for a much simpler theory of $p$-adic families of automorphic forms: see \cite{TheseG}, \cite{Loe}.
Several papers already rely on Theorem \ref{theo:main} for groups $\mathbf{G}$ satisfying this stronger assumption: \cite{Taideform} (and thus also \cite{CaraLeHung}), \cite{ChRe} (see Conjectures 3.26 and 3.30), and \cite{CheLan} (see Conjecture 4.25).

The multiplicity formula for arbitrary inner forms of symplectic and special orthogonal groups over number fields was announced by Arthur in \cite[Chapter 9]{Arthur}, but even formulating the result precisely was not possible at the time because of the absence of canonical absolute transfer factors as in the quasi-split case, which were only recently constructed by Kaletha \cite{Kal}, \cite{Kalglob}.
Furthermore, the general case will require the construction of local Arthur packets for non-quasi-split groups, as well as solving the problem, observed by Arthur, that some inner forms of even special orthogonal groups do not admit an outer automorphism defined over the base field.
In contrast, the above assumptions on $\mathbf{G}$ allow us to give a simple proof of the multiplicity formula, thanks to the following simplifications.
\begin{itemize}
\item If $\mathbf{G}$ is an even special orthogonal group, i.e.\ a \emph{pure} inner form of a quasi-split group rather than an arbitrary inner form, then the non-trivial outer automorphism of $\mathbf{G}$ has a representative defined over $F$.
In fact there is a natural $\mathbf{G}(F)$-orbit of such automorphisms.
\item At the finite places of $F$, we can use Arthur packets for quasi-split groups constructed in \cite{Arthur}.
\item The assumption on $\mathbf{G}$ at places in $S$ and the ``algebraic regular'' assumption imply that the only relevant Arthur-Langlands parameters at these places are those considered by Adams and Johnson in \cite{AdJo}, that is parameters bounded on the Weil group of $\R$ and having \emph{algebraic regular} infinitesimal character.
This simplifies significantly the spectral side in the stabilisation of the trace formula, since we do not have to consider contributions from \emph{non-discrete} parameters to the discrete part of the spectral side of the stabilisation of the trace formula as in the general case.
\end{itemize}
For applications, using arbitrary inner forms instead of quasi-split groups is constraining for at least two reasons.

The first reason is that at the non-quasi-split places it is in general difficult to describe precisely the local Arthur packets, or even to establish whether a given Arthur packet is non-empty, contrary to the quasi-split case where Arthur packets contain the corresponding Langlands packets by \cite[Proposition 7.4.1]{Arthur}.
This is not an issue in our case, as we will make the internal parametrisation of Adams-Johnson packets explicit using Shelstad's parametrisation of tempered L-packets (\cite{She1}, \cite{She2}, \cite{She3}) and Kaletha's refinement of Shelstad's formulation in \cite[§5.6]{Kal}.
We will be able to appeal to \cite{AdJo} thanks to recent work of Arancibia, Moeglin and Renard \cite{AMR} which proves that for Adams-Johnson parameters and quasi-split symplectic or special orthogonal groups, Arthur packets coincide with Adams-Johnson packets.
In particular these packets are non-empty.

The second reason that makes arbitrary inner forms less practical is that contrary to the quasi-split case it is not true that any \emph{tempered} Arthur-Langlands parameter gives rise to at least one automorphic representation, namely the generic one (with respect to a choice of Whittaker datum).
Again the fact that Arthur packets at the real places of $F$ are completely explicit can allow to circumvent this difficulty in situations where solvable base change is possible, as in \cite{Taideform}.

Since Mok \cite{Mokunitary} did for quasi-split unitary groups what Arthur did for quasi-split symplectic and special orthogonal groups, and Arancibia, Moeglin and Renard also treated the case of unitary groups in \cite{AMR}, the same method as in the present paper could be used to prove the analogue of Theorem \ref{theo:main} for inner forms of unitary groups for which there exists a non-empty set $S$ of real places of $F$ satisfying the above assumption.
In fact this case would be easier since there are no issues with outer automorphisms as in the even orthogonal case.
However, for \emph{pure} inner forms of unitary groups and \emph{tempered} parameters Kaletha, Minguez, Shin and White have already proved Arthur's multiplicity formula in \cite{KalMinShiWhi}, and they also announced a proof of the general case.
They used a slightly different formulation, using Kaletha's extended pure inner twists instead of Kaletha's rigid inner twists that we use in the present paper.
We refer the interested reader to \cite{Kalrivsbg} for a comparison of these notions.

Let us end this introduction with notations that we will use throughout the paper.
For $K$ a perfect field we will denote by $\mathrm{Gal}_K$ the absolute Galois group $\mathrm{Gal}(\overline{K} / K)$ of $K$.
If $K$ is a local or global field of characteristic zero, $W_K$ will denote its Weil group.
If $K$ is a local field of characteristic zero, $\mathrm{WD}_K$ will denote its Weil-Deligne group: $\mathrm{WD}_K = W_K$ if $K \simeq \R$ or $\C$, and $\mathrm{WD}_K = W_K \times \mathrm{SU}(2)$ if $K$ is $p$-adic.

\section{Review of Arthur's results}
\label{sec:reviewArthur}

In this section we briefly review Arthur's results from \cite{Arthur}, i.e.\ the construction of local and adélic Arthur packets for \emph{quasi-split} special orthogonal and symplectic groups.

\subsection{Quasi-split symplectic and special orthogonal groups and their Langlands dual groups}
\label{sec:defgroups}

Let $F$ be a global or local field of characteristic zero.
As in \cite{Arthur}, we will consider three families of quasi-split reductive groups over $F$:
\begin{itemize}
\item for $n \geq 0$, the split symplectic group $\mathbf{Sp}_{2n}$, defined as the stabilizer of a non-degenerate alternating form on a vector space of dimension $2n$ over $F$,
\item for $n \geq 1$, the split special orthogonal group $\mathbf{SO}_{2n+1}$, defined as follows.
Let $V$ be a vector space of dimension $2n+1$ over $F$, and let $q$ be a maximally split non-degenerate quadratic form $q$ on $V$: $(V,q)$ is the orthogonal direct sum of $n$ hyperbolic planes and a line.
Then $\mathbf{SO}_{2n+1}$ is the stabilizer of $q$ in the special linear group $\mathbf{SL}(V)$.
\item for $n \geq 1$ and $\alpha \in F^{\times} / F^{\times 2}$, the quasi-split special orthogonal group $\mathbf{SO}_{2n}^{\alpha}$, defined as follows.
Let $E = F[X] / (X^2-\alpha)$ and consider the quadratic form $q_{\alpha} : x \mapsto N_{E / F}(x)$ on $E$ seen as a two-dimensional vector space over $F$.
Note that $(E, q_{\alpha})$ is a hyperbolic plane if and only if $\alpha = 1$, i.e.\ $\alpha$ is a square.
Let $(V,q)$ be the orthogonal direct sum of $n-1$ hyperbolic planes and $(E, q_{\alpha})$.
Then $\mathbf{SO}_{2n}^{\alpha}$ is the stabilizer of $q$ in the special linear group $\mathbf{SL}(V)$.
\end{itemize}
In the third case we will also see $\alpha$ as a continuous character $\A_F^{\times} / F^{\times} \rightarrow \{ \pm 1 \}$ if $F$ is global ($\alpha : F^{\times} \rightarrow \{ \pm 1 \}$ if $F$ is local), using class field theory.

Denote by $\mathbf{G}$ one of the above groups.
As in \cite{BorelCorvallis} denote by ${}^L \mathbf{G}$ the Langlands dual group of $\mathbf{G}$ and $\widehat{\mathbf{G}}$ its identity component, so that
$$ \widehat{\mathbf{G}} = \begin{cases} \mathrm{SO}_{2n+1}(\C) & \text{if } \mathbf{G} = \mathbf{Sp}_{2n} \\ \mathrm{Sp}_{2n}(\C) & \text{if } \mathbf{G} = \mathbf{SO}_{2n+1} \\ \mathrm{SO}_{2n}(\C) & \text{if } \mathbf{G} = \mathbf{SO}_{2n}^{\alpha} \end{cases} $$
and ${}^L \mathbf{G} = \widehat{\mathbf{G}} \rtimes \Gal_F$.
The action of $\Gal_F$ on $\widehat{\mathbf{G}}$ is trivial except when $\mathbf{G} = \mathbf{SO}_{2n}^{\alpha}$ with $\alpha$ nontrivial, in which case the action of $\Gal_F$ factors through $\Gal(E / F) = \{ 1, \sigma \}$ and $\sigma$ acts by outer conjugation on $\widehat{\mathbf{G}}$.
To be completely explicit, one can identify $\widehat{\mathbf{G}} \rtimes \mathrm{Gal}(E / F)$ with $\mathrm{O}_{2n}(\C)$ as follows.
Let $B$ be the non-degenerate symmetric bilinear form on $\C v_1 \oplus \dots \oplus \C v_{2n}$ defined by $B(v_i, v_j) = \delta_{i, 2n+1-j}$, then $1 \rtimes \sigma$ is identified with the element of $\mathrm{O}_{2n}(\C)$ fixing $v_1, \dots, v_{n-1}, v_{n+2}, \dots , v_{2n}$ and exchanging $v_n$ and $v_{n+1}$.

Denote by $\mathrm{Std}_{\mathbf{G}}$ the standard representation $\widehat{\mathbf{G}} \rightarrow \mathrm{GL}_N(\C)$ where
$$ N = N(\widehat{\mathbf{G}}) = \begin{cases} 2n & \text{if } \mathbf{G}=\mathbf{SO}_{2n+1} \text{ or } \mathbf{SO}_{2n}^{\alpha}, \\ 2n+1 & \text{if } \mathbf{G}=\mathbf{Sp}_{2n}. \end{cases} $$
Note that we have included the trivial group $\mathbf{Sp}_0$ in the list above because the standard representation of its dual group has dimension $1$.
Extend $\mathrm{Std}_{\mathbf{G}}$ to ${}^L \mathbf{G}$ as follows.
\begin{itemize}
\item If $\mathbf{G}$ is symplectic, odd orthogonal, or split even orthogonal (that is $\alpha=1$), let $\mathrm{Std}_{\mathbf{G}}$ be trivial on $\Gal_F$.
\item Otherwise $\mathbf{G} = \mathbf{SO}_{2n}^{\alpha}$ with $\alpha \neq 1$.
Let $\mathrm{Std}_{\mathbf{G}}$ be trivial on $\Gal_E$.
If $n=1$, up to conjugation by $\mathrm{Std}_{\mathbf{G}}(\widehat{\mathbf{G}})$ there is a unique choice for $\mathrm{Std}_{\mathbf{G}}(1 \rtimes \sigma)$.
If $n>1$, there are two possibilities for $\mathrm{Std}_{\mathbf{G}}(1 \rtimes \sigma)$, choose the one having $-1$ as an eigenvalue with multiplicity one (the other choice has $-1$ as an eigenvalue with multiplicity $N-1 = 2n-1$).
Via the above identification $\widehat{\mathbf{G}} \rtimes \mathrm{Gal}(E / F) \simeq \mathrm{O}_{2n}(\C)$, this choice for  $\mathrm{Std}_{\mathbf{G}}$ is simply the natural inclusion $\mathrm{O}_{2n}(\C) \subset \mathrm{GL}_{2n}(\C)$.
\end{itemize}

Finally, let $\mathrm{Aut}({}^L \mathbf{G})$ be the group of automorphisms of ${}^L \mathbf{G}$, and $\mathrm{Out}({}^L \mathbf{G}) = \mathrm{Aut}({}^L \mathbf{G}) / \widehat{\mathbf{G}}_{\mathrm{ad}}$.
Note that $\mathrm{Out}({}^L \mathbf{G})$ is trivial unless $\mathbf{G} = \mathbf{SO}_{2n}^{\alpha}$, in which case it is isomorphic to $\Z / 2 \Z$.
Even in this case, the $\mathrm{GL}_N(\C)$-conjugacy class of $\mathrm{Std}_{\mathbf{G}}$ is invariant under $\mathrm{Out}({}^L \mathbf{G})$.

\subsection{Self-dual automorphic cuspidal representations of $\mathbf{GL}_N$}

In this section we let $F$ be a number field.
Let $N \geq 1$, and consider an automorphic cuspidal representation $\pi$ of $\mathbf{GL}_N/F$ which is self-dual.
Arthur \cite[Theorem 1.4.1]{Arthur} associates a quasi-split special orthogonal or symplectic group $\mathbf{G}_{\pi}$ with $\pi$, which satisfies $N(\widehat{\mathbf{G}_{\pi}}) = N$.
The central character $\omega_{\pi}$ of $\pi$ has order 1 or 2, and is trivial if $\widehat{\mathbf{G}_{\pi}}$ is symplectic.
If $N$ is odd, $\pi \otimes \omega_{\pi}$ is also self-dual and has trivial central character, and $\mathbf{G} = \mathbf{Sp}_{N-1}$.
Choose $\mathrm{Std}_{\pi} : {}^L \mathbf{G}_{\pi} \rightarrow \mathrm{GL}_N(\C)$ extending the standard representation $\widehat{\mathbf{G}_{\pi}} \rightarrow \mathrm{GL}_N(\C)$:
\begin{itemize}
\item If $\widehat{\mathbf{G}_{\pi}}$ is odd orthogonal, $\mathrm{Std}_{\pi}$ is the twist of $\mathrm{Std}_{\mathbf{G_{\pi}}}$ (defined in the previous section) by the character $\mathrm{Gal}_F \rightarrow \{ \pm 1 \}$ corresponding to $\omega_{\pi}$ by class field theory.
\item Otherwise let $\mathrm{Std}_{\pi} = \mathrm{Std}_{\mathbf{G}}$.
\end{itemize}
We can now state the remaining part of \cite[Theorem 1.4.1 and 1.4.2]{Arthur}: for any place $v$ of $F$, the local Langlands parameter $\mathrm{WD}_{F_v} \rightarrow \mathrm{GL}_N(\C)$ of $\pi_v$ is conjugate to $\mathrm{Std}_{\pi} \circ \phi_v$ for some Langlands parameter $\phi_v : \mathrm{WD}_{F_v} \rightarrow {}^L \mathbf{G}_{\pi}$.
Moreover the $\mathrm{Aut} \left( {}^L \mathbf{G}_{\pi} \right)$-orbit of $\phi_v$ is determined by the Langlands parameter of $\pi_v$.

Define $\mathrm{sign}(\widehat{\mathbf{G}}) = +1$ (resp.\ $-1$) if $\widehat{\mathbf{G}}$ is orthogonal (resp.\ symplectic), and define $\mathrm{sign}(\pi) = \mathrm{sign}(\widehat{\mathbf{G}_{\pi}})$.

\subsection{Elliptic endoscopic data and embeddings of L-groups}
\label{sec:ellenddat}

In this section $F$ denotes a local or global field of characteristic zero, and $\mathbf{G}$ denotes one of the groups defined in section \ref{sec:defgroups}
Let us recall from \cite[§1.8]{WaldFormulaire} the isomorphism classes of elliptic endoscopic data for $\mathbf{G}$, and fix an embedding of L-groups in each case.
To avoid confusion we ought to be precise about the terminology that we will use, as there is a slight discrepancy between \cite{KS} and \cite{WaldFormulaire}.
We will use the definition of an endoscopic datum $(\mathbf{H}, \mathcal{H}, s, \xi)$ for $\mathbf{G}$ given in \cite[§2.1]{KS}.
In particular, $\mathcal{H}$ is a splittable extension of $W_F$ by $\widehat{\mathbf{H}}$ and $\xi : \mathcal{H} \rightarrow {}^L \mathbf{G}$ is an L-embedding.
Waldspurger also includes an L-embedding ${}^L \xi : {}^L \mathbf{H} \rightarrow {}^L \mathbf{G}$ which is the composition of $\xi : \mathcal{H} \rightarrow {}^L \mathbf{G}$ and an isomorphism ${}^L \mathbf{H} \rightarrow \mathcal{H}$.
As observed in \cite[§2.2]{KS}, in general such an isomorphism does not necessarily exist, a problem which is overcome by introducing a z-extension of $\mathbf{H}$.
Fortunately such an isomorphism does exist for any endoscopic datum of any group $\mathbf{G}$ as in section \ref{sec:defgroups}, and thus we need not consider such z-extensions.

Here is the list of all isomorphism classes of elliptic endoscopic data for the groups defined in section \ref{sec:defgroups}.
For simplicity we only give $\mathbf{H}$, and refer to \cite[§1.8]{WaldFormulaire} for more details.
We also give the group $\mathrm{Out}(\mathbf{H}, \mathcal{H}, s, \xi) = \mathrm{Aut}(\mathbf{H}, \mathcal{H}, s, \xi)/\widehat{\mathbf{H}}$ of outer automorphisms in each case.
\begin{itemize}
\item $\mathbf{G} = \mathbf{SO}_{2n+1}$: $\mathbf{H} = \mathbf{SO}_{2a+1} \times \mathbf{SO}_{2b+1}$ with $a+b=n$.
The pairs $(a,b)$ and $(b,a)$ define isomorphic endoscopic data.
The outer automorphism group of $(\mathbf{H}, \mathcal{H}, s, \xi)$ is trivial except when $a=b$, where there is a unique non-trivial outer automorphism, swapping the two factors of $\widehat{\mathbf{H}} = \mathrm{Sp}_{2a}(\C) \times \mathrm{Sp}_{2b}(\C)$.
\item $\mathbf{G} = \mathbf{Sp}_{2n}$: $\mathbf{H} = \mathbf{Sp}_{2a} \times \mathbf{SO}_{2b}^{\alpha}$ with $a+b=n$ and $(b,\alpha) \neq (1,1)$.
We impose that $\alpha=1$ if $b=0$.
There is a non-trivial outer automorphism if and only if $b>0$, in which case the unique non-trivial automorphism acts by outer conjugation on the factor $\mathrm{SO}_{2b}(\C)$ of $\widehat{\mathbf{H}}$.
\item $\mathbf{G} = \mathbf{SO}_{2n}^{\alpha}$: $\mathbf{H} = \mathbf{SO}_{2a}^{\beta} \times \mathbf{SO}_{2b}^{\gamma}$, with $a+b=n$, $\beta = 1$ if $a=0$, $\gamma=1$ if $b=0$, $\beta \gamma = \alpha$, and both $(a, \beta)$ and $(b, \gamma)$ are distinct from $(1,1)$.
The pairs $((a,\beta),(b,\gamma))$ and $((b,\gamma),(a,\beta))$ define isomorphic endoscopic data.
If $ab>0$ there is a non-trivial outer automorphism acting by simultaneous outer conjugation on the two factors of $\widehat{\mathbf{H}}$.
There are no other non-trivial outer automorphisms except when $a=b$ and $\alpha=1$, in which case there is one swapping the two factors of $\widehat{\mathbf{H}}$, and so $\mathrm{Out}(\mathbf{H}, \mathcal{H}, s, \xi) \simeq \Z / 2\Z \times \Z / 2\Z$ in this case.
\end{itemize}
As explained above, in each case Waldspurger also chooses a particular L-embedding ${}^L \xi : {}^L \mathbf{H} \rightarrow {}^L \mathbf{G}$.
This choice is somewhat arbitrary, as it could be twisted by a 1-cocycle $W_F \rightarrow Z(\widehat{\mathbf{H}})$, and all results requiring the theory of endoscopy (for example the endoscopic character relation \ref{eqn:endcharrelqs}) ought to be valid for any choice of embedding.
However, as in \cite{Arthur} it will be convenient to fix these embeddings, at least in the global setting, and we refer to \cite[§1.8]{WaldFormulaire} for their definition.
It is important to note that Waldspurger's choice of embedding ${}^L \xi$, a priori defined only for a particular representative $(\mathbf{H}, \mathcal{H}, s, \xi)$ in its isomorphism class of elliptic endoscopic data, has the benefit of being invariant under $\mathrm{Out}(\mathbf{H}, \mathcal{H}, s, \xi)$.
Consequently this gives a well-defined embedding ${}^L \xi'$ for any endoscopic datum $(\mathbf{H}', \mathcal{H}', s', \xi')$ isomorphic to $(\mathbf{H}, \mathcal{H}, s, \xi)$.

It is not difficult to check directly on the definition of transfer factors that the validity of statements such as \ref{eqn:endcharrelqs} does not depend on the choice of embedding.

Finally, let us determine the relationship between Waldspurger's choice of embeddings and the standard representations defined in \ref{sec:defgroups}, which are embeddings for twisted endoscopy for general linear groups.
Write $\mathbf{H} = \mathbf{H}_1 \times \mathbf{H}_2$.
If $\mathbf{G}$ is orthogonal or $\mathbf{G} = \mathbf{Sp}_{2n}$ and $\mathbf{H} = \mathbf{Sp}_{2a} \times \mathbf{SO}_{2b}^{\alpha}$ with $\alpha=1$, then $\mathrm{Std}_{\mathbf{G}} \circ {}^L \xi$ is conjugated to $\left(\mathrm{Std}_{\mathbf{H}_1} \oplus \mathrm{Std}_{\mathbf{H}_2}\right) \circ \iota$, where $\iota$ denotes the inclusion ${}^L \mathbf{H} \subset {}^L \mathbf{H}_1 \times {}^L \mathbf{H}_2$.
In the remaining case where $\mathbf{G} = \mathbf{Sp}_{2n}$ and $\mathbf{H} = \mathbf{Sp}_{2a} \times \mathbf{SO}_{2b}^{\alpha}$ with $\alpha \neq 1$, $\mathrm{Std}_{\mathbf{G}} \circ {}^L \xi$ is conjugated to $\left(\left( \alpha \otimes \mathrm{Std}_{\mathbf{Sp}_{2a}} \right) \oplus \mathrm{Std}_{\mathbf{SO}_{2b}^{\alpha}}\right) \circ \iota$, where we see $\alpha$ as a quadratic character of $W_F$.

\subsection{Arthur's substitute for global parameters}
\label{sec:subsALparam}

In this section we let $F$ be a number field.
In order to formulate his multiplicity formula \cite[Theorem 1.5.2]{Arthur}, Arthur circumvented the absence of the hypothetical Langlands group by introducing substitutes for Arthur-Langlands parameters for special orthogonal and symplectic groups.
Consider formal, unordered sums $\psi = \boxplus_i \pi_i [d_i]$ where $\pi_i$ is a self-dual automorphic cuspidal representation of $\mathbf{GL}_{N_i} / F$ and $d_i \geq 1$ is an integer representing the dimension of an irreducible algebraic representation of $\mathrm{SL}_2(\C)$.
Let $\mathbf{G}$ denote a quasi-split special orthogonal or symplectic group over $F$ as in section \ref{sec:defgroups}, and let $N = N(\widehat{\mathbf{G}})$.
Let $\widetilde{\Psi}_{\mathrm{disc}}(\mathbf{G})$ be the set of $\psi = \boxplus_i \pi_i [d_i]$ as above such that the pairs $(\pi_i, d_i)$ are distinct, $\sum_i N_i d_i = N$, for all $i$ $\mathrm{sign}(\pi_i)(-1)^{d_i-1} = \mathrm{sign}(\widehat{\mathbf{G}})$, and
\begin{itemize}
\item if $\mathbf{G}$ is symplectic, that is if $\widehat{\mathbf{G}}$ is odd orthogonal, $\prod_i \omega_{\pi_i}^{d_i} = 1$,
\item if $\mathbf{G} = \mathbf{SO}_{2n}^{\alpha}$, $\prod_i \omega_{\pi_i}^{d_i} = \alpha$.
\end{itemize}
Note that if $\mathbf{G}$ is odd orthogonal, that is if $\widehat{\mathbf{G}}$ is symplectic, the condition $\prod_i \omega_{\pi_i}^{d_i} = 1$ is automatically satisfied.

Following Arthur, for a formal sum $\psi = \boxplus_i \pi_i [d_i]$ such that the pairs $(\pi_i, d_i)$ are distinct we let $\mathcal{L}_{\psi}$ be the fibre product of $\left({}^L \mathbf{G}_{\pi_i} \rightarrow \mathrm{Gal}_F \right)_i$.
Recall that for each $i$ we have $\mathrm{Std}_{\pi_i} : {}^L \mathbf{G}_{\pi_i} \rightarrow \mathrm{GL}_{N_i}(\C)$, and so we can form the representation $\bigoplus_i \mathrm{Std}_{\pi_i} \otimes \nu_{d_i}$ of $\mathcal{L}_{\psi} \times \mathrm{SL}_2(\C)$ where $\nu_d$ is the unique irreducible algebraic representation of $\mathrm{SL}_2(\C)$ in dimension $d$.
Then $\psi \in \widetilde{\Psi}_{\mathrm{disc}}(\mathbf{G})$ if and only there exists $\dot{\psi} : \mathcal{L}_{\psi} \times \mathrm{SL}_2(\C) \rightarrow {}^L \mathbf{G}$ such that $\mathrm{Std}_{\mathbf{G}} \circ \dot{\psi}$ is conjugated to $\bigoplus_i \mathrm{Std}_{\pi_i} \otimes \nu_{d_i}$.
Moreover in this case $\dot{\psi}$ is unique up to conjugation by $\widehat{\mathbf{G}}$ except if $\mathbf{G}$ is even orthogonal and all $N_i d_i$ are even, in which case there are exactly two conjugacy classes of such $\dot{\psi}$, exchanged by the group $\mathrm{Aut} \left( {}^L \mathbf{G} \right) / \widehat{\mathbf{G}}$ which has two elements.
For $\psi \in \widetilde{\Psi}_{\mathrm{disc}}(\mathbf{G})$ let $m_{\psi}$ be the number of $\widehat{\mathbf{G}}$-conjugacy classes of such $\dot{\psi}$, so that $m_{\psi} \in \{ 1,2 \}$.
Finally, denote by $\Psi_{\mathrm{disc}}(\mathbf{G})$ the set of such pairs $(\psi, \dot{\psi})$.
To simplify the notation we will abusively denote by $\dot{\psi}$ such an element of $\Psi_{\mathrm{disc}}(\mathbf{G})$.

Let us recall some definitions from \cite[§10.2]{KottSTFcusp}.
For $\dot{\psi} \in \Psi_{\mathrm{disc}}(\mathbf{G})$ we let $C_{\dot{\psi}} = \mathrm{Cent}(\dot{\psi}, \widehat{\mathbf{G}}) \supset Z(\widehat{\mathbf{G}})^{\mathrm{Gal}_F}$.
It is easy to check that $C_{\dot{\psi}}$ is a finite $2$-group.
We also let $S_{\dot{\psi}}$ be the group of $g \in \widehat{G}$ such that $g \dot{\psi}(x) g^{-1} \dot{\psi}(x)^{-1} \in Z(\widehat{\mathbf{G}})$ for all $x \in \mathcal{L}_{\psi} \times \mathrm{SL}_2(\C)$ and the resulting 1-cocycle $W_F \rightarrow Z(\widehat{\mathbf{G}})$ is locally trivial.
In particular, $S_{\dot{\psi}}$ contains $Z(\widehat{\mathbf{G}})$.
For the groups $\mathbf{G}$ considered in this paper the action of $\mathrm{Gal}_F$ on $Z(\widehat{\mathbf{G}})$ factors through a cyclic extension, thus any locally trivial 1-cocycle $W_F \rightarrow Z(\widehat{\mathbf{G}})$ is trivial, and so $S_{\dot{\psi}} = C_{\dot{\psi}} Z(\widehat{\mathbf{G}})$.
In fact the action of $\mathrm{Gal}_F$ on $Z(\widehat{\mathbf{G}})$ is trivial except in the uninteresting case $\mathbf{G} = \mathbf{SO}_2^{\alpha}$ with $\alpha \neq 1$, so that $S_{\dot{\psi}} = C_{\dot{\psi}}$ except in this case.
Finally, let $\mathcal{S}_{\dot{\psi}} = S_{\dot{\psi}} / Z(\widehat{\mathbf{G}})$.
By the above remark, we also have $\mathcal{S}_{\dot{\psi}} = C_{\dot{\psi}} / Z(\widehat{\mathbf{G}})^{\mathrm{Gal}_F}$.
The three groups $C_{\dot{\psi}}$, $S_{\dot{\psi}}$ and $\mathcal{S}_{\dot{\psi}}$ can be canonically described using only $\psi$, in fact using only the family $(N_i d_i \mod 2)_i$ if $\mathbf{G}$ is fixed, and so to lighten our notations we will simply denote them by $C_{\psi}$, $S_{\psi}$ and $\mathcal{S}_{\psi}$ in the sequel.

Let $(\mathbf{H}, \mathcal{H}, s, \xi)$ be an elliptic endoscopic datum for $\mathbf{G}$.
Recall that in section \ref{sec:ellenddat} we fixed an L-isomorphism ${}^L \mathbf{H} \simeq \mathcal{H}$, and thus an embedding ${}^L \xi : {}^L \mathbf{H} \rightarrow {}^L \mathbf{G}$.
Recall also that $\mathbf{H} = \mathbf{H}_1 \times \mathbf{H}_2$ where $\mathbf{H}_1$ and $\mathbf{H}_2$ are also quasi-split special orthogonal or symplectic groups.
Define $\widetilde{\Psi}_{\mathrm{disc}}(\mathbf{H}) = \widetilde{\Psi}_{\mathrm{disc}}(\mathbf{H}_1) \times \widetilde{\Psi}_{\mathrm{disc}}(\mathbf{H}_2)$ and similarly for $\Psi_{\mathrm{disc}}(\mathbf{H})$.
For $\psi' = (\psi'_1, \psi'_2) = (\boxplus_{i \in I_1} \pi_i[d_i], \boxplus_{i \in I_2} \pi_i[d_i]) \in \widetilde{\Psi}_{\mathrm{disc}}(\mathbf{H})$ it is natural to define
$$ {}^L \xi(\psi') = \begin{cases} \boxplus_{i \in I_1 \sqcup I_2} \pi_i[d_i] & \text{if } \mathbf{G} = \mathbf{SO}_{2n+1} \text{ or } \mathbf{SO}_{2n}^{\alpha}, \\ \left( \boxplus_{i \in I_1} (\alpha \otimes \pi_i)[d_i] \right) \boxplus \left( \boxplus_{i \in I_2} \pi_i[d_i] \right) & \text{if } \mathbf{G} = \mathbf{Sp}_{2n},\ \mathbf{H}_1 = \mathbf{Sp}_{2a} \text{ and } \mathbf{H}_2 = \mathbf{SO}_{2b}^{\alpha}. \end{cases} $$
Clearly ${}^L \xi(\psi') \in \widetilde{\Psi}_{\mathrm{disc}}(\mathbf{G})$ if and only if the pairs $(\pi_i, d_i)_{i \in I_1 \sqcup I_2}$ are distinct.
Define $\widetilde{\Psi}_{\mathbf{G}-\mathrm{disc}}(\mathbf{H}) \subset \widetilde{\Psi}_{\mathrm{disc}}(\mathbf{H})$ as the subset of elements satisfying this condition, and define $\Psi_{\mathbf{G}-\mathrm{disc}}(\mathbf{H})$ similarly.
For $\psi' = (\psi'_1, \psi'_2) \in \widetilde{\Psi}_{\mathbf{G}-\mathrm{disc}}(\mathbf{H})$ we can identify the fiber product of $\mathcal{L}_{\psi'_1} \rightarrow \mathrm{Gal}_F$ and $\mathcal{L}_{\psi'_2} \rightarrow \mathrm{Gal}_F$ with $\mathcal{L}_{{}^L \xi(\psi')}$.
Note that in the case $\mathbf{G} = \mathbf{Sp}_{2n}$, this identification comes from identifying $\mathbf{G}_{\pi_i}$ and $\mathbf{G}_{\alpha \otimes \pi_i}$ for $i \in I_1$.
Thus we can define, for $\dot{\psi}' \in \Psi_{\mathbf{G}-\mathrm{disc}}(\mathbf{H})$ above $\psi'$, its image ${}^L \xi(\dot{\psi}') \in \Psi_{\mathrm{disc}}(\mathbf{G})$ above ${}^L \xi(\psi')$.
Moreover $s \in S_{{}^L \xi(\dot{\psi}')}$.

Conversely, for $\dot{\psi} \in \Psi_{\mathrm{disc}}(\mathbf{G})$, any $s \in S_{\psi}$ defines an elliptic endoscopic datum $(\mathbf{H}, \mathcal{H}, s, \xi)$ where $\xi(\widehat{\mathbf{H}}) = \mathrm{Cent}(s, \widehat{\mathbf{G}})^0$ and $\xi(\mathcal{H}) = \xi(\widehat{\mathbf{H}}) \dot{\psi}(\mathcal{L}_{\psi})$.
The isomorphism class of $(\mathbf{H}, \mathcal{H}, s, \xi)$ clearly only depends on the image of $s$ in $\mathcal{S}_{\psi}$, and there is a unique $\dot{\psi}' \in \Psi_{\mathbf{G}-\mathrm{disc}}(\mathbf{H})$ such that $\dot{\psi} = {}^L \xi(\dot{\psi}')$.
We obtain the following proposition.
\begin{prop}[Arthur] \label{prop:endobij}
The mappings $\dot{\psi}' \mapsto {}^L \xi(\dot{\psi}')$ induce a bijection between isomorphism classes of pairs $((\mathbf{H}, \mathcal{H}, s, \xi), \dot{\psi}')$ where $(\mathbf{H}, \mathcal{H}, s, \xi)$ is an elliptic endoscopic datum for $\mathbf{G}$ and $\dot{\psi}' \in \Psi_{\mathbf{G}-\mathrm{disc}}(\mathbf{H})$, and $\widehat{\mathbf{G}}$-conjugacy classes of pairs $(\dot{\psi}, s)$ where $\dot{\psi} \in \Psi_{\mathrm{disc}}(\mathbf{G})$ and $s \in \mathcal{S}_{\psi}$.
\end{prop}

\begin{rema}
The statement of the Proposition is particularly simple thanks to the fact that $\mathbf{G}$ and $\mathbf{H}$ satisfy the Hasse principle (see \cite[§4]{KottSTFcusp}).
In general the ``right'' notion to consider is not that of a parameter $\dot{\psi}$ up to $\widehat{\mathbf{G}}$-conjugacy, but also up to locally trivial elements of $Z^1(W_F, Z(\widehat{\mathbf{G}}))$, and similarly for $\mathbf{H}$.
Then an equivalence class of $\dot{\psi}'$ can correspond to more than one equivalence class of $(\dot{\psi}, s)$.
We refer the interested reader to \cite[§11]{KottSTFcusp}, particularly Proposition 11.2.1 for the general case.
\end{rema}

\subsection{Local Arthur packets in the quasi-split case}
\label{sec:Apackqs}

In this section we let $F$ be a local field of characteristic zero, and $\mathbf{G}$ one of the reductive groups over $F$ defined in section \ref{sec:defgroups}.
If $F$ is $p$-adic, let $\mathcal{H}(\mathbf{G})$ be the Hecke algebra of smooth compactly supported distributions on $\mathbf{G}$ with complex coefficients.
If $F$ is Archimedean, fix a maximal compact subgroup $K$ of $\mathbf{G}(F)$ and let $\mathcal{H}(\mathbf{G})$ be the algebra of bi-$K$-finite smooth compactly supported distributions on $\mathbf{G}(F)$ with complex coefficients.

Contrary to the global case, we now have genuine Arthur-Langlands parameters at our disposal, i.e.\ continuous morphisms $\psi : \mathrm{WD}_F \times \mathrm{SL}_2(\C) \rightarrow {}^L \mathbf{G}$ such that
\begin{itemize}
\item the composition with ${}^L \mathbf{G} \rightarrow W_F$ is the obvious projection, 
\item for any $w \in \mathrm{WD}_F$, $\psi(w)$ is semisimple,
\item the restriction of $\psi$ to the factor $\mathrm{SL}_2(\C)$ is algebraic.
\end{itemize}
We let $\Psi(\mathbf{G})$ be the set of Arthur-Langlands parameters for $\mathbf{G}$ such that $\psi(\mathrm{WD}_F)$ is bounded.

To formulate Arthur's construction of packets associated with such parameters precisely, it is necessary to fix a Whittaker datum $\mathfrak{w}$ for $\mathbf{G}$.
If $F$ is $p$-adic and $\mathbf{G}$ is unramified, there is a unique $\mathbf{G}(F)$-conjugacy class of hyperspecial maximal compact subgroups compatible with $\mathfrak{w}$ in the sense of \cite{CasSha}.
In this case ``unramified representation of $\mathbf{G}(F)$'' will mean unramified with respect to such a subgroup.

Arthur's construction relies on the local Langlands correspondence for general linear groups, via the standard representation $\mathrm{Std}_{\mathbf{G}}$ of ${}^L \mathbf{G}$.
For this reason, the following fact is important: the $\mathrm{Aut}({}^L \mathbf{G})$-orbit of an Arthur-Langlands parameter $\psi$ is determined by the conjugacy class of $\mathrm{Std}_{\mathbf{G}} \circ \psi$.
Recall that for $\mathbf{G} = \mathbf{SO}_{2n}^{\alpha}$ the group $\mathrm{Out}({}^L \mathbf{G})$ has two elements, and so Arthur's methods using the twisted trace formula for general linear groups cannot distinguish between packets associated with $\widehat{\mathbf{G}}$-conjugacy classes of parameters in the same $\mathrm{Out}({}^L \mathbf{G})$-orbit.
If $\mathbf{G} = \mathbf{SO}_{2n}^{\alpha}$, there exists an outer automorphism $\theta$ of $\mathbf{G}$ which preserves $\mathfrak{w}$: in fact $\theta$ can be realised as an element of the corresponding orthogonal group of determinant $-1$, as one can easily check for $\mathbf{SO}_4^{\alpha}$, and the general case follows.
If $F$ is Archimedean, the maximal compact subgroup $K$ of $\mathbf{G}(F)$ can be chosen to be stable under $\theta$.
Following Arthur, let $\widetilde{\mathcal{H}}(\mathbf{G})$ be the subalgebra of $\mathcal{H}(\mathbf{G})$ consisting of $\theta$-invariant distributions, so that irreducible representations of $\widetilde{\mathcal{H}}(\mathbf{G})$ correspond to $\{1, \theta \}$-orbits of irreducible representations.
Note that it is not really necessary to consider $\theta$-invariant distributions, distributions whose orbital integrals are $\theta$-invariant suffice: if $\pi$ is an admissible representation of $\mathbf{G}(F)$ and $f(g)dg \in \mathcal{H}(\mathbf{G})$ has $\theta$-invariant orbital integrals, then $\mathrm{tr} \left(\pi(f(g)dg)\right) = \mathrm{tr} \left(\pi^{\theta}(f(g)dg)\right)$.
If $\mathbf{G}$ is symplectic or odd-orthogonal, we simply let $\widetilde{\mathcal{H}}(\mathbf{G}) = \mathcal{H}(\mathbf{G})$ and $\theta=1$.

As in the global case, denote by $C_{\psi}$ the centraliser of $\psi$ in $\widehat{\mathbf{G}}$, $S_{\psi} = Z(\widehat{\mathbf{G}}) C_{\psi}$ and
$$\mathcal{S}_{\psi} = \pi_0\left( C_{\psi} / Z(\widehat{\mathbf{G}})^{\mathrm{Gal}_F} \right) = \pi_0\left( S_{\psi} / Z(\widehat{\mathbf{G}}) \right) $$
which is an abelian 2-group.
Note that the image $s_{\psi}$ of $-1 \in \mathrm{SL}_2(\C)$ by $\psi$ belongs to $C_{\psi}$.

Arthur \cite[Theorem 1.5.1]{Arthur} associates with any $\psi \in \Psi(\mathbf{G})$ a multi-set $\Pi_{\psi}$ of $\{1, \theta \}$-orbits of irreducible unitary representations $\pi$ of $\mathbf{G}(F)$, along with a map $\Pi_{\psi} \rightarrow \mathcal{S}_{\psi}^{\vee}$, denoted by $\pi \mapsto \langle \cdot, \pi \rangle$.
Arthur proves that the linear form
\begin{equation} \label{eqn:defstable}
\Lambda_{\psi} : f(g)dg \mapsto \sum_{\pi \in \Pi_{\psi}} \langle s_{\psi}, \pi \rangle \mathrm{tr} \left( \pi(f(g)dg)\right)
\end{equation} 
is \emph{stable}, i.e.\ vanishes if all stable orbital integrals of $f(g)dg \in \widetilde{\mathcal{H}}(\mathbf{G})$ vanish.
These Arthur packets are characterised by \cite[Theorem 2.2.1]{Arthur}.
Note that our notation $\Lambda_{\psi}(f(g)dg)$ differs from Arthur's $f^{\mathbf{G}}(\psi)$.

Let us recall the endoscopic character relations \cite[Theorem 2.2.1(b)]{Arthur}.
Any semisimple $s \in S_{\psi}$ determines an endoscopic datum $(\mathbf{H}, \mathcal{H}, s, \xi)$ for $\mathbf{G}$.
For any choice of L-embedding ${}^L \xi : {}^L \mathbf{H} \rightarrow {}^L \mathbf{G}$, there is a unique Arthur-Langlands parameter $\psi'$ for $\mathbf{H}$ such that $\psi = {}^L \xi \circ \psi'$.
Note that $\mathbf{H}$ is a product of general linear groups and quasi-split special orthogonal or symplectic groups, and so $\psi'$ can be seen as a collection of Arthur-Langlands parameters for these groups.
The endoscopic datum $(\mathbf{H}, \mathcal{H}, s, \xi)$, the L-embedding ${}^L \xi$ and the Whittaker datum $\mathfrak{w}$ allow to define \emph{transfer factors} \cite[§5.3]{KS}, and thus to define the notion of distributions on $\mathbf{G}(F)$ and $\mathbf{H}(F)$ having matching orbital integrals (\cite[§5.5]{KS}, also called \emph{transfer}).
Any $f(g)dg \in \widetilde{\mathcal{H}}(\mathbf{G})$ admits a transfer $f'(h)dh \in \widetilde{\mathcal{H}}(\mathbf{H})$, and
\begin{equation} \label{eqn:endcharrelqs}
\sum_{\pi \in \Pi_{\psi}} \langle s_{\psi} s, \pi \rangle \mathrm{tr} \left( \pi(f(g)dg)\right) = \Lambda_{\psi'}(f'(h)dh) = \sum_{\pi' \in \Pi_{\psi'}} \langle s_{\psi'}, \pi' \rangle \mathrm{tr} \left( \pi'(f'(h)dh)\right).
\end{equation}

In a global setting it is also necessary to consider Arthur-Langlands parameters which are not bounded on $\mathrm{WD}_F$, because the generalised Ramanujan conjecture is not known to hold in general.
In this case the packet $\Pi_{\psi}$ is simply defined by parabolic induction from the bounded case, so that its elements are not a priori irreducible.
We refer to \cite[§1.5]{Arthur} for details.

\subsection{Localisation of global parameters}
\label{sec:localisationpsi}

Finally, let us explain the relation between the previous two sections.
Let $F$ be a number field, and consider $(\psi, \dot{\psi}) \in \Psi_{\mathrm{disc}}(\mathbf{G})$.
Write $\psi$ as $\boxplus_i \pi_i[d_i]$.
Let $v$ be a place of $F$.
For any $i$ the Langlands parameter of the representation $\pi_{i,v}$ of $\mathbf{GL}_{N_i}(F_v)$ factors through $\mathrm{Std}_{\pi_i} : {}^L \mathbf{G}_{\pi_i} \rightarrow \mathrm{GL}_{N_i}(\C)$.
We obtain a Langlands parameter $\mathrm{WD}_{F_v} \rightarrow {}^L \mathbf{G}_{\pi}$, well-defined up to the action of $\mathrm{Aut}({}^L \mathbf{G}_{\pi_i})$.
Composing with $\dot{\psi}$ we get an Arthur-Langlands parameter $\dot{\psi}_v : \mathrm{WD}_{F_v} \times \mathrm{SL}_2(\C) \rightarrow {}^L \mathbf{G}$.
Moreover we have natural morphisms $C_{\dot{\psi}} \rightarrow C_{\dot{\psi}_v}$, $S_{\dot{\psi}} \rightarrow S_{\dot{\psi}_v}$ and $\mathcal{S}_{\dot{\psi}} \rightarrow \mathcal{S}_{\dot{\psi}_v}$.

\section{Certain rigid inner forms and Arthur packets}

\subsection{Certain rigid inner forms}

Let $F$ be a totally real number field.
From now on we use $\mathbf{G}^*$ to denote a quasi-split special orthogonal or symplectic group over $F$ as in section \ref{sec:defgroups}, as we will use $\mathbf{G}$ to denote inner forms of $\mathbf{G}^*$.
We shall be interested in certain inner twists of $\mathbf{G}^*$, i.e.\ pairs $(\mathbf{G}, \Xi)$ where $\mathbf{G}$ is a reductive group over $F$ and $\Xi : \mathbf{G}^*_{\overline{F}} \rightarrow \mathbf{G}_{\overline{F}}$ is an isomorphism such that for any $\sigma \in \Gal_F$, the automorphism $\Xi^{-1} \sigma(\Xi)$ of $\mathbf{G}^*_{\overline{F}}$ is inner.
For $\sigma \in \Gal_F$, denote by $z_{\sigma}$ the element of $\mathbf{G}^*_{\mathrm{ad}}(\overline{F})$ such that $\Xi^{-1} \sigma(\Xi) = \mathrm{Ad}(z_{\sigma})$.
This defines a cocyle $z_{\cdot}$ in $Z^1(F, \mathbf{G}^*_{\mathrm{ad}})$, and the cohomology set $H^1(F, \mathbf{G}^*_{\mathrm{ad}})$ classifies the isomorphism classes of inner twists of $\mathbf{G}^*$.
Assume that $F$ is not totally complex and let $S$ be a non-empty set of real places of $F$.
We say that $(\mathbf{G}, S)$ satisfies property $\star$ if
\begin{itemize}
\item for any place $v \in S$, the group $\mathbf{G}(F_v)$ admits discrete series,
\item for any place $v \not\in S$, the reductive group $\mathbf{G}_{F_v}$ is quasi-split.
\end{itemize}
We will classify such pairs, and realise them as \emph{rigid inner twists} (see \cite{Kal}, \cite{Kalglob}).
In fact in the even orthogonal case we will restrict to \emph{pure inner forms}, for reasons explained in Remark \ref{rema:onlypureevenorth}. 

\subsubsection{Rigid inner forms at real places}
\label{sec:innerformsreal}

In this section we only consider a fixed real place $v$ of $F$, and to simplify notation we denote without subscript the base change of $\mathbf{G}^*$ to $F_v = \R$, and write $\mathrm{Gal}(\C / \R) = \{ 1, \tau \}$.
We are interested in inner forms $\mathbf{G}$ of $\mathbf{G}^*$.
It is well-known that $\mathbf{G}(\R)$ admits discrete series if and only if $\mathbf{G}^*(\R)$ admits discrete series, and this condition is also equivalent to $\mathbf{G}^*$ admitting an anisotropic maximal torus, or to $\mathbf{G}^*(\R)$ admitting an anisotropic inner form.
Moreover there is at most one element of $H^1(\R, \mathbf{G}^*_{\mathrm{ad}})$ giving rise to an anisotropic inner form.

Below we describe in which cases $\mathbf{G}^*(\R)$ admits discrete series and recall the classification of its inner forms, restricting to \emph{pure} inner forms in the even orthogonal case.
In the orthogonal cases we will make these inner forms explicit by realising them as a \emph{pure} inner forms.
More precisely we will choose a pair $(\mathbf{B}^*, \mathbf{T}^*)$ where $\mathbf{T}^*$ is an anisotropic maximal torus of $\mathbf{G}^*$ and $\mathbf{B}^* \supset \mathbf{T}^*_{\C}$ is a Borel subgroup of $\mathbf{G}^*_{\C}$ such that all simple roots for $(\mathbf{B}^*, \mathbf{T}^*)$ are non-compact, and explicitly compute the surjective map $H^1(\R, \mathbf{T}^*) \rightarrow H^1(\R, \mathbf{G}^*)$.
This will be useful in section \ref{sec:Apacketsreal} to compute the internal parametrisation of tempered discrete packets and Adams-Johnson packets for inner forms of $\mathbf{G}^*$.
In the symplectic case the non-trivial inner forms of $\mathbf{G}^*$ cannot be realised as pure inner forms, which is why the notion of rigid inner twist \cite{Kal} is needed.
We denote by $u_{\R} \rightarrow \mathcal{E}_{\R}$ the extension of $\mathrm{Gal}(\C / \R)$ that was denoted by $u \rightarrow W$ in \cite{Kal}, to avoid confusion with the Weil group.
We will fix a pair $(\mathbf{B}^*, \mathbf{T}^*)$ as in the orthogonal case and compute the surjective mappings $$H^1(u_{\R} \rightarrow \mathcal{E}_{\R}, \mathbf{Z}(\mathbf{G}^*) \rightarrow \mathbf{T}^*) \rightarrow H^1(u_{\R} \rightarrow \mathcal{E}_{\R}, \mathbf{Z}(\mathbf{G}^*) \rightarrow \mathbf{G}^*) \rightarrow H^1(\R, \mathbf{G}^*_{\mathrm{ad}}). $$
\begin{itemize}
\item Case $\mathbf{G}^* = \mathbf{SO}_{2n+1} = \mathbf{SO}(V,q)$.
The group $\mathbf{G}^*$ is split and adjoint.
The cohomology set $H^1(\R, \mathbf{G}^*)$ parametrises quadratic spaces having same dimension and discriminant as $(V,q)$.
Let $\mathbf{T}^*$ be an anisotropic maximal torus of $\mathbf{G}^*$.
It corresponds to an orthogonal decomposition $(V,q) = (D,q_0) \oplus \bigoplus_{j=1}^n (P_j,q_j)$ where $D$ is a line and each $(P_j,q_j)$ is an anisotropic plane: $\mathbf{T}^*$ is the identity component of the stabilizer of this decomposition and $\mathbf{T}^* \simeq \prod_j \mathbf{SO}(P_j, q_j)$.
Assume that $q_0$ is positive, then there are $\lceil n/2 \rceil$ negative (resp.\ $\lfloor n/2 \rfloor$ positive) planes among the $P_j$'s.
Thus up to reordering we can assume that $(-1)^{n+1-j} q_j$ is positive for any $j$.
For any $j$, fix a basis $e_j$ of the free $\Z$-module $X^*(\mathbf{SO}(P_j, q_j))$ of rank one.
There is a unique Borel subgroup $\mathbf{B}^*$ of $\mathbf{G}^*_{\C}$ containing $\mathbf{T}^*_{\C}$ such that the set of simple roots of $\mathbf{T}^*_{\C}$ in $\mathbf{B}^*$ is $\{ e_1-e_2, \dots, e_{n-1}-e_n, e_n \}$.
This determines an identification $\widehat{\mathbf{T}^*} \simeq \mathcal{T}$ where $\mathcal{T}$ is a maximal torus of $\widehat{\mathbf{G}^*} \simeq \mathrm{Sp}_{2n}(\C)$, part of a pair $(\mathcal{B}, \mathcal{T})$ such that the set of simple roots of $\mathcal{T}$ in $\mathcal{B}$ is $\{ e_1^*-e_2^*, \dots, e_{n-1}^*-e_n^*, 2e_n^* \}$ where $e_j^*(e_k) = \delta_{j,k}$.
The group $H^1(\R, \mathbf{T}^*)$ is in perfect duality with $\widehat{\mathbf{T}^*}^{\mathrm{Gal}(\C / \R)}$ which is identified with the subgroup of $2$-torsion in $\mathcal{T}$ because $\tau$ acts by inversion on $\widehat{\mathbf{T}^*}$.
It image of $(\epsilon_j)_{1 \leq j \leq n} \in \{ \pm 1 \}^n \simeq H^1(\R, \mathbf{T}^*)$ in $H^1(\R, \mathbf{G}^*)$ corresponds to the isomorphism class of inner twists $(\mathbf{G}, \Xi)$ where $\mathbf{G}$ is a special orthogonal group of signature
$$(a,b) = \left(1 + 2 \mathrm{card}\left\{ j\,|\, \epsilon_j = (-1)^{n+1-j} \right\}, 2 \mathrm{card}\left\{ j\,|\, \epsilon_j = (-1)^{n-j} \right\}\right).$$
In particular we see that the character of $Z(\widehat{\mathbf{G}^*})^{\mathrm{Gal}(\C / \R)} = \{ \pm 1 \}$ that \cite[Theorem 1.2]{KottEllSing} associates to the image of $(\epsilon_j)_{1 \leq j \leq n}$ in $H^1(\R, \mathbf{G}^*)$ is $-1 \mapsto \prod_{1 \leq j \leq n} \epsilon_j = (-1)^{n(n-1)/2 + (a-1)/2}$.

There is a unique element of $H^1(\R, \mathbf{T}^*)$ giving rise to the anisotropic inner form of $\mathbf{G}^*$, namely $((-1)^{n+1-j})_{1 \leq j \leq n}$.
The associated character of $Z(\widehat{\mathbf{G}^*})^{\mathrm{Gal}(\C / \R)}$ is $-1 \mapsto (-1)^{n(n+1)/2}$.
\item Case $\mathbf{G}^* = \mathbf{SO}_{2n}^{\alpha} = \mathbf{SO}(V,q)$, where $\alpha \in \R^{\times} / \R_{>0} \simeq \{ \pm 1 \}$ is equal to $(-1)^n \mathrm{disc}(q)$.
The group $\mathbf{G}^*$ admits discrete series if and only if $(-1)^n \alpha > 0$.
Assume that this is the case.
Let $\mathbf{T}^*$ be an anisotropic maximal torus of $\mathbf{G}^*$.
It corresponds to an orthogonal decomposition $(V,q) = \bigoplus_{j=1}^n (P_j,q_j)$ where each $(P_j,q_j)$ is an anisotropic plane: $\mathbf{T}^*$ is the identity component of the stabilizer of this decomposition and $\mathbf{T}^* \simeq \prod_j \mathbf{SO}(P_j, q_j)$.
There are $\lceil n/2 \rceil$ positive (resp.\ $\lfloor n/2 \rfloor$ negative) planes among the $(V_j, q_j)$'s.
Up to reordering we can assume that $(-1)^{j-1} q_j$ is positive definite for all $j$.
As in the previous case fix bases $e_j$ of $X^*(\mathbf{SO}(P_j, q_j))$, and let $\mathbf{B}^*$ be the Borel subgroup corresponding to the set of simple roots $\{ e_1-e_2, \dots, e_{n-1}-e_n, e_{n-1} + e_n \}$, and the set of simple roots for the pair $(\mathcal{B}, \mathcal{T})$ in $\widehat{\mathbf{G}^*} \simeq \mathrm{SO}_{2n}(\C)$ is $\{ e_1^*-e_2^*, \dots, e_{n-1}^*-e_n^*, e_{n-1}^* + e_n^* \}$.
As before $H^1(\R, \mathbf{T}^*) \simeq \{ \pm 1 \}^n$ is in perfect duality with the subgroup of $2$-torsion in $\mathcal{T}$.
The set $H^1(\R, \mathbf{G}^*)$ parametrises the isomorphism classes of quadratic forms on dimension $2n$ vector spaces over $\R$ having positive discriminant.
The image of $(\epsilon_j)_{1 \leq j \leq n} \in \{ \pm 1 \}^n \simeq H^1(\R, \mathbf{T}^*)$ in $H^1(\R, \mathbf{G}^*)$ corresponds to the signature
$$(a,b) = \left(2 \mathrm{card}\left\{ j\,|\, \epsilon_j = (-1)^{j-1} \right\}, 2 \mathrm{card}\left\{ j\,|\, \epsilon_j = (-1)^j \right\}\right) $$
and the associated character of $Z(\widehat{\mathbf{G}^*})^{\mathrm{Gal}(\C / \R)} = \{ \pm 1 \}$ is
$$-1 \mapsto \prod_{1 \leq j \leq n} \epsilon_j = (-1)^{n(n+1)/2 + a/2} .$$
The fibres of the mapping $H^1(\R, \mathbf{G}^*) \rightarrow H^1(\R, \mathbf{G}^*_{\mathrm{ad}})$ are given by the signatures $(a,b)$ and $(b,a)$.

There are two elements of $H^1(\R, \mathbf{T}^*)$ giving rise to anisotropic inner forms of $\mathbf{G}^*$, namely $((-1)^{j-1})_{1 \leq j \leq n}$ and $((-1)^j)_{1 \leq j \leq n}$.
They map to distinct classes in $H^1(\R, \mathbf{G}^*)$, and the associated characters of $Z(\widehat{\mathbf{G}^*})^{\mathrm{Gal}(\C / \R)}$ are $-1 \mapsto (-1)^{n(n-1)/2}$ and $-1 \mapsto (-1)^{n(n+1)/2}$.
They map to the same class in $H^1(\R, \mathbf{G}^*_{\mathrm{ad}})$.

\item Case $\mathbf{G}^* = \mathbf{Sp}_{2n} = \mathbf{Sp}(V, a)$ where $V$ is a $2n$-dimensional vector space over $\R$ and $a$ is a non-degenerate alternating form.
In this case $\mathbf{G}^*$ always admits discrete series.
We have that $H^1(\R, \mathbf{G}^*) = 1$ but for $n>0$ the set $H^1(\R, \mathbf{G}^*_{\mathrm{ad}})$ is non-trivial.
First we describe the anisotropic inner form as a rigid inner form.
Choose $J \in \mathbf{G}^*(\R)$ such that $J^2 = -1$ and for any $v \in V \smallsetminus \{0\}$, $a(Jv, v)>0$.
Choose $i \in \C$ such that $i^2=-1$.
These define a complex structure on $V$, along with a positive definite hermitian form
$$ h(v_1, v_2) := a(Jv_1, v_2) + i a(v_1, v_2). $$
Choose a decomposition $(V,h) = \bigoplus_{j=1}^n (P_j, h_j)$ where each $P_j$ is a complex line.
Then the stabiliser $\mathbf{T}^*$ of this decomposition is an anisotropic maximal torus of $\mathbf{G}^*$.
We have a canonical isomorphism $f_j : \mathbf{U}(P_j, h_j) \simeq \mathbf{U}_1$ for each $j$.
Fix a basis $b$ of $X^*(\mathbf{U}_1)$, and let $e_j = (-1)^{j-1} b \circ f_j$, a basis of $X^*(\mathbf{U}(P_j, h_j))$.
Let $\mathbf{B}^* \supset \mathbf{T}^*_{\C}$ be the Borel subgroup of $\mathbf{G}^*_{\C}$ having as set of simple roots $\{ e_1 - e_2, \dots, e_{n-1}-e_n, 2e_n \}$.
The basis $(e_1, \dots, e_n)$ identifies $X_*(\mathbf{T}^*)$ with $\Z^n$.
Let $\mathbf{Z} \simeq \mathbf{\mu}_2$ be the center of $\mathbf{G}^*$, and let $\overline{\mathbf{T}}^* = \mathbf{T}^* / \mathbf{Z}$, so that $X_*(\overline{\mathbf{T}}^*) = \Z^n + \Z (\frac{1}{2}, \dots, \frac{1}{2})$.
There are two elements of $H^1(u_{\R} \rightarrow \mathcal{E}_{\R}, \mathbf{Z} \rightarrow \mathbf{T}^*) \simeq X_*(\overline{\mathbf{T}}^*) / 2 X_*(\mathbf{T}^*)$ (\cite[Theorem 4.8]{Kal}) giving rise to the compact inner form of $\mathbf{G}^*$: $z_1 = \mathrm{cl}\left(\frac{1}{2}, -\frac{1}{2}, \dots, \frac{(-1)^{n-1}}{2}\right)$ and $z_2 = \mathrm{cl}\left(-\frac{1}{2}, \frac{1}{2}, \dots, \frac{(-1)^n}{2}\right)$.
Note that $X_*(\overline{\mathbf{T}}^*) / 2 X_*(\mathbf{T}^*)$ is naturally isomorphic to the group of characters of the preimage of $\widehat{\mathbf{T}^*}^{\mathrm{Gal}(\C / \R)}$ in $\widehat{\overline{\mathbf{T}}^*}$.
The elements $z_1, z_2$ map to the same element in $H^1(\R, \overline{\mathbf{T}}^*) \simeq X_*(\overline{\mathbf{T}}^*) / 2X_*(\overline{\mathbf{T}}^*)$, namely the one represented by $\tau \mapsto J$.
Let $(\mathbf{G}, \Xi)$ be a corresponding inner twist of $\mathbf{G}^*$.
The positive definite quadratic form $v \mapsto a(Jv, v)$ on $V$ yields a positive definite hermitian form $h'$ on $\C \otimes_{\R} V$ seen as a complex vector space using the tensor product (not $J$!).
It is easy to check that $\mathbf{G}$ is a Zariski-closed subgroup of the anisotropic unitary group $\mathbf{U}(\C \otimes_{\R} V, h')$, and thus $\mathbf{G}$ is the anisotropic inner form of $\mathbf{G}^*$.
In fact we can give $\C \otimes_{\R} V$ a right $\mathbb{H}$-vector space structure by letting $j \in \mathbb{H}$ act by $\lambda \otimes v \mapsto \bar{\lambda} \otimes J(v)$, and $\mathbf{G}$ is the intersection of $\mathbf{U}(\C \otimes_{\R} V, h')$ with the real algebraic group of automorphisms of this $\mathbb{H}$-vector space.

Recall \cite[§3.4]{Kal} the quotient $H^1_{\mathrm{sc}}(u_{\R} \rightarrow \mathcal{E}_{\R}, \mathbf{Z} \rightarrow \mathbf{G}^*)$ of the pointed set $H^1(u_{\R} \rightarrow \mathcal{E}_{\R}, \mathbf{Z} \rightarrow \mathbf{G}^*)$.
This quotient has a natural structure of abelian group.
By \cite[Theorem 4.11 and Proposition 5.3]{Kal} (which generalise \cite[Theorem 1.2]{KottEllSing}), $H^1_{\mathrm{sc}}(u_{\R} \rightarrow \mathcal{E}_{\R}, \mathbf{Z} \rightarrow \mathbf{G}^*) \simeq Z(\widehat{\mathbf{G}^*}_{\mathrm{sc}})^{\vee} \simeq \Z / 2 \Z$, and the image of $z_i$ is the non-trivial element, for any $i$.
For any $i$, the fibre of this non-trivial element can be identified with $\psi^{-1}\left( H^1(\R, \mathbf{G}) \right) z_i$, and thus
$$ H^1(u_{\R} \rightarrow \mathcal{E}_{\R}, \mathbf{Z} \rightarrow \mathbf{G}^*) = \{1\} \bigsqcup \psi^{-1}\left( H^1(\R, \mathbf{G}) \right) z_i. $$
The pointed set $H^1(\R, \mathbf{G})$ can be computed using \cite[Chapitre III, Théorème 6]{SerreGalCo}, and it is in natural bijection with the set of conjugacy classes of elements $y$ of order $1$ or $2$ in $\mathbf{G}(\R)$ (equivalently, $\mathbf{G}(\C)$).
Mapping $y$ to $\dim \ker (y + \mathrm{Id}) / 2$ (in the standard representations of $\mathbf{G}(\C)$) allows to identify $H^1(\R, \mathbf{G})$ with the integer interval $\{ 0, \dots, n \}$.
The trivial element of $H^1(\R, \mathbf{G})$ corresponds to $0$ and the element $x$ such that $\psi^{-1}(x) z_i = z_{3-i}$ corresponds to $n$.
The elements of $\psi^{-1} \left(H^1(\R, \mathbf{G}) \right) z_i$ corresponding to integers $k$ and $k'$ map to the same class in $H^1(\R, \mathbf{G}^*_{\mathrm{ad}})$ if and only if $k' \in \{k, n-k\}$.
The trivial class in $H^1(u_{\R} \rightarrow \mathcal{E}_{\R}, \mathbf{Z} \rightarrow \mathbf{G}^*)$ is the only one mapping to the trivial class in $H^1(\R, \mathbf{G}^*_{\mathrm{ad}})$.
In particular $H^1(\R, \mathbf{G}^*)$ has $\lfloor n/2 \rfloor + 2$ elements.
\end{itemize}

\begin{rema} \label{rema:onlypureevenorth}
The reader might wonder why we have not considered rigid inner forms instead of only pure inner forms in the even orthogonal case.
For $n \geq 2$ denote $\mathbf{G}^* = \mathbf{SO}_{2n}^{(-1)^n}$.
It is not difficult to compute that the complement of $H^1(\R, \mathbf{G}^*)$ in $H^1(u_{\R} \rightarrow \mathcal{E}_{\R}, \mathbf{Z}(\mathbf{G}^*) \rightarrow \mathbf{G}^*)$ has two elements.
They map to the same element in $H^1(\R, \mathbf{G}^*_{\mathrm{ad}})$ if and only if $n$ is odd.
These two elements are swapped by the outer automorphism of $\mathbf{G}^*$.
For this reason we cannot formulate Proposition \ref{prop:realAJArthur} for inner forms of even special orthogonal groups which are not pure inner forms, and thus we cannot formulate Theorem \ref{theo:main} either.
Note that for $n$ is even, any non-pure inner form of $\mathbf{SO}_{2n}^1$ does not admit any outer automorphism defined over $\R$. 
\end{rema}

\subsubsection{Global rigid inner forms}
\label{sec:gri}

We can now classify which groups among the quasi-split groups $\mathbf{G}^*$ over a number field $F$ defined in section \ref{sec:defgroups} admit inner forms $\mathbf{G}$ satisfying property $\star$ with respect to a given non-empty set $S$ of real places of $F$.
As explained in the previous remark, we restrict ourselves to \emph{pure} inner forms in the even orthogonal case.
Since odd special orthogonal groups are adjoint, we only have to consider pure inner forms when $\mathbf{G}^*$ is orthogonal.
In the symplectic case we will realise all inner forms $\mathbf{G}$ of $\mathbf{G}^*$ satisfying property $\star$ with respect to $S$ as \emph{rigid inner forms}.
We refer to \cite{Kal} (resp.\ \cite{Kalglob}) for the construction of the extension $u_v \rightarrow \mathcal{E}_v$ of $\mathrm{Gal}_{F_v}$ (resp.\ $P_{F, \dot{V}} \rightarrow \mathcal{E}_{F, \dot{V}}$ of $\mathrm{Gal}_F$) by a profinite algebraic group, and the cohomology sets for reductive groups defined using these extensions.
Note that as in the previous section we denote by $u_v \rightarrow \mathcal{E}_v$ the extension that is denoted by $u_v \rightarrow W_v$ in \cite{Kal}, to avoid confusion with the Weil group, and that we suppressed the choice of $\tilde{\xi}$ in the notation $\mathcal{E}_{F, \dot{V}}$.
Let $\mathbf{Z}$ denote the center of $\mathbf{G}^*$.

Of course up to removing places from $S$ we only have to consider groups $\mathbf{G}^*$ that are not quasi-split at all places in $S$.

\begin{prop} \label{prop:gri}
Let $F$ be a totally real number field, $S$ a non-empty set of real places of $F$, and $n \geq 1$.
\begin{enumerate}
\item $\mathbf{G}^* = \mathbf{SO}_{2n+1} = \mathbf{SO}(V,q)$ where $q$ has positive discriminant at all places of $S$.
For each place $v \in S$ choose $z_v \in H^1(F_v, \mathbf{G}^*)$, corresponding to a signature $(1+2a_v, 2b_v)$ where $a_v+b_v=n$, with associated character of $Z(\widehat{\mathbf{G}^*})^{\mathrm{Gal}_{F_v}} \simeq \{ \pm 1 \}$:
$$-1 \mapsto \epsilon_v = (-1)^{n(n-1)/2+a_v}. $$
Then $\mathbf{G}^*$ admits an inner form satisfying property $\star$ with respect to $S$ if and only if $\prod_{v \in S} \epsilon_v = 1$.
In this case it is unique.
\item $\mathbf{G}^* = \mathbf{SO}_{2n}^{\alpha}$ where $\alpha \in F^{\times} / F^{\times 2}$ and $\mathbf{G}^* = \mathbf{SO}(V, q)$ where $q$ is defined in section \ref{sec:defgroups}.
Assume that $(-1)^n \alpha$ is positive at all places in $S$.
For each $v \in S$ choose $z_v \in H^1(F_v, \mathbf{G}^*)$, corresponding to the signature $(2a_v, 2 b_v)$ where $a_v+b_v=n$ and with associated character of $Z(\widehat{\mathbf{G}^*})^{\mathrm{Gal}_{F_v}} \simeq \{ \pm 1 \}$:
$$-1 \mapsto \epsilon_v = (-1)^{n(n+1)/2+a_v}. $$
\begin{enumerate}
\item Assume that $n$ is odd.
Then $\mathbf{G}^*$ has a unique inner form $\mathbf{G}$ satisfying property $\star$ with respect to $S$ and isomorphic to the inner form given by the image of $z_v$ in $H^1(F_v, \mathbf{G}^*_{\mathrm{ad}})$ at all $v \in S$.

Let $V$ be a finite set of real or finite places of $F$ disjoint from $S$ and such that for any $v \in V$, $\alpha \not\in F_v^{\times 2}$, and moreover $(-1)^{\mathrm{card} (V)} = \prod_{v \in S} \epsilon_v$.
For any $v \in V$ let $z_v \in H^1(F_v, \mathbf{G}^*)$ be the unique non-trivial class mapping to the trivial class in $H^1(F_v, \mathbf{G}^*)$.
Explicitly, $z_v$ corresponds to the signature $(n-1, n+1)$ if $v$ is real, and $z_v$ is the unique non-trivial element of $H^1(F_v, \mathbf{G}^*)$ if $v$ is finite.
Then there is a unique $z \in H^1(F, \mathbf{G}^*)$ such that for any $v \in S \cup V$ the image of $z$ in $H^1(F_v, \mathbf{G}^*)$ equals $z_v$, and for any $v \not\in S \cup V$ the image of $z$ in $H^1(F_v, \mathbf{G}^*)$ is trivial.
The corresponding pure inner form of $\mathbf{G}^*$ is isomorphic to $\mathbf{G}$ as an inner form, and this describes  all the pure inner forms of $\mathbf{G}^*$ satisfying this property.
\item Assume that $n$ is even.
Then $\mathbf{G}^*$ has an inner form $\mathbf{G}$ satisfying property $\star$ with respect to $S$ and isomorphic to the image of $z_v$ in $H^1(F_v, \mathbf{G}^*_{\mathrm{ad}})$ at all places $v \in S$ if and only if $\prod_{v \in S} \epsilon_v = 1$ or $\alpha \neq 1$.
If it exists then it is unique up to isomorphism.

Assume that $\prod_{v \in S} \epsilon_v = 1$ or $\alpha \neq 1$.
Let $V$ be a finite set of real or finite places of $F$ disjoint from $S$ and such that for any $v \in V$, $\alpha \not\in F_v^{\times 2}$, and moreover $(-1)^{\mathrm{card} (V)} = \prod_{v \in S} \epsilon_v$.
Such a set exists and $V = \emptyset$ is the only possible choice if $\alpha = 1$.
For any $v \in V$ let $z_v \in H^1(F_v, \mathbf{G}^*)$ be the unique non-trivial class mapping to the trivial class in $H^1(F_v, \mathbf{G}^*)$.
Explicitly, $z_v$ corresponds to the signature $(n-1, n+1)$ if $v$ is real, and $z_v$ is the unique non-trivial element of $H^1(F_v, \mathbf{G}^*)$ if $v$ is finite.
Then there is a unique $z \in H^1(F, \mathbf{G}^*)$ such that for any $v \in S \cup V$ the image of $z$ in $H^1(F_v, \mathbf{G}^*)$ equals $z_v$, and for any $v \not\in S \cup V$ the image of $z$ in $H^1(F_v, \mathbf{G}^*)$ is trivial.
The corresponding pure inner form of $\mathbf{G}^*$ is isomorphic to $\mathbf{G}$ as an inner form, and this describes  all the pure inner forms of $\mathbf{G}^*$ satisfying this property.
\end{enumerate}
\item $\mathbf{G}^* = \mathbf{Sp}_{2n}$.
For each $v \in S$ choose $z_v \in H^1(F_v, \mathbf{G}^*_{\mathrm{ad}}) \smallsetminus \{ 1 \}$.
Then $\mathbf{G}^*$ has an inner form satisfying property $\star$ with respect to $S$ and which is isomorphic to the image of $z_v$ at each place $v \in S$ if and only if $\mathrm{card}(S)$ is even.
If it exists then it is unique up to isomorphism.

Let $\mathbf{Z}$ be the center of $\mathbf{G}^*$.
For each $v \in S$ choose $\tilde{z}_v \in H^1(u_v \rightarrow \mathcal{E}_v, \mathbf{Z} \rightarrow \mathbf{G}^*) \smallsetminus \{ 1 \}$ lifting $z_v$.
Up to isomorphism there is a unique global rigid inner form which is isomorphic to $z_v$ at any $v \in S$ and split at all places of $F$ not in $S$.
\end{enumerate}
\end{prop}
\begin{proof}
Note that in all cases the groups $\mathbf{G}^*_{\mathrm{ad}}$ and $\mathbf{G}^*$ and all their inner forms satisfy the Hasse principle by \cite[§4]{KottSTFcusp}, since the action of $\mathrm{Gal}_F$ on the center of the dual group factors through a cyclic group.
This fact implies all claims of unicity appearing in the Proposition.

The first two cases of special orthogonal groups are simple exercises using the previous section, \cite{KottEllSing} and \cite[Theorem 6.6]{PlaRap}, as is the first part of the symplectic case.
For the last part of the third case we need Kaletha's generalisation of Kottwitz' results.
Assume that $\mathbf{G}^* = \mathbf{Sp}_{2n}$ and that $\mathrm{card}(S)$ is even.
Theorem 3.44 and Corollary 3.45 in \cite{Kalglob} describe the image of localisation $H^1(P_{F, \dot{V}} \rightarrow \mathcal{E}_{F, \dot{V}}, \mathbf{Z} \rightarrow \mathbf{G}^*) \rightarrow \prod_v H^1(u_v \rightarrow \mathcal{E}_v, \mathbf{Z} \rightarrow \mathbf{G}^*)$ and implies the existence of at least one relevant $z \in H^1(P_{F, \dot{V}} \rightarrow \mathcal{E}_{F, \dot{V}}, \mathbf{Z} \rightarrow \mathbf{G}^*)$.
By \cite[Lemma 3.22]{Kalglob} the image of $z$ by $H^1(P_{F, \dot{V}} \rightarrow \mathcal{E}_{F, \dot{V}}, \mathbf{Z} \rightarrow \mathbf{G}^*) \rightarrow \mathrm{Hom}(P_{\dot{V}}, \mathbf{Z})$ is imposed, and so two such relevant elements $z_1, z_2$ differ by an element of $H^1(F, \mathbf{G}_1) \simeq \prod_{v \in S} H^1(F_v, \mathbf{G}_1)$, where $\mathbf{G}_1$ is the inner form of $\mathbf{G}^*$ determined by $z_1$. 
The sets $H^1(F_v, \mathbf{G}_1)$, $H^1(F_v, \mathbf{G}_{1, \mathrm{ad}})$ and the mappings between them were made explicit in the previous section.
\end{proof}
Note that pure inner forms are special cases of rigid inner forms by letting $\mathbf{Z}=1$, thus all cases could be formulated using rigid inner forms.
It is not difficult to check that any inner form of an even special orthogonal group which is anisotropic at all Archimedean places of $F$ and quasi-split at all finite places of $F$ arises as a pure inner forms, i.e.\ is a special orthogonal group.
In the symplectic case the inner forms considered in the Proposition can be constructed explicitly as unitary groups over a quaternion algebra over $F$ which is non-split exactly at the places in $S$.

\begin{rema} \label{rema:probtriv}
We will use this realisation of an inner form $\mathbf{G}$ as a rigid inner form of $\mathbf{G}^*$ to get a coherent family of normalisations of the local Langlands correspondences for the groups $\mathbf{G}(F_v)$, following \cite{Kal}.
Here ``coherent'' can be intuitively understood as ``satisfying a product formula'', a condition that is necessary to formulate the main goal of this paper, i.e.\ Arthur's multiplicity formula for automorphic representations of $\mathbf{G}$ (Theorem \ref{theo:main}).
In the case $\mathbf{G}^* = \mathbf{SO}_{2n+1}$ (resp.\ $\mathbf{G}^* = \mathbf{Sp}_{2n}$) for any $z \in H^1(F, \mathbf{G}^*)$ (resp.\ $H^1(P_{F, \dot{V}} \rightarrow \mathcal{E}_{F, \dot{V}}, \mathbf{Z} \rightarrow \mathbf{G}^*)$) as in the Proposition and any place $v \not\in S$, the image of $z$ in $H^1(F_v, \mathbf{G}^*)$ (resp.\ $H^1(u_v \rightarrow \mathcal{E}_v, \mathbf{Z} \rightarrow \mathbf{G}^*)$) is trivial.
This is not always the case for $\mathbf{G}^* = \mathbf{SO}_{2n}^{\alpha}$.
In fact it can happen that no relevant $z \in H^1(F, \mathbf{G}^*)$ is locally trivial at all places of $F$ not in $S$, i.e.\ that $V$ as in the Proposition is necessarily non-empty: this occurs when $\prod_{v \in S} \epsilon_v = - 1$.
In these cases it is not possible to rely solely on Arthur's construction of packets reviewed in section \ref{sec:Apackqs}.
Fortunately Kaletha calculated the effect of twisting the local Langlands correspondence by an element of $H^1(F_v, \mathbf{Z}(\mathbf{G}))$, and we will review his results in section \ref{sec:trivit}.
\end{rema}

\subsection{Local Arthur packets for real reductive groups having discrete series}
\label{sec:Apacketsreal}

Let $\mathbf{G}^*$ be a quasi-split reductive group as in section \ref{sec:defgroups} over $\R$, and assume that $\mathbf{G}^*$ admits discrete series.
Let $\mathbf{G}$ be an inner form of $\mathbf{G}^*$.
Let $\mathbf{T}^*$ be an anisotropic maximal torus of $\mathbf{G}^*$, and let $\mathbf{Z}$ be the center of $\mathbf{G}^*$, except for $\mathbf{G}^* = \mathbf{SO}_2^{\alpha}$ (with $\alpha = -1$) which is a torus, in which case we let $\mathbf{Z}$ be any finite subgroup of $\mathbf{G}^*$.
As we saw in section \ref{sec:innerformsreal}, the inner form $\mathbf{G}$ of $\mathbf{G}^*$ can be realised as a rigid inner twist $(\mathbf{G}, \Xi, z)$ of $\mathbf{G}^*$, i.e.\ $\Xi : \mathbf{G}^*_{\C} \rightarrow \mathbf{G}_{\C}$ is an isomorphism and $z \in Z^1(u_{\R} \rightarrow \mathcal{E}_{\R}, \mathbf{Z} \rightarrow \mathbf{G}^*)$ is such that for any $w \in \mathcal{E}_{\R}$, $\Xi^{-1} w(\Xi) = \mathrm{Ad}(z(w))$.
In fact we realised $z$ as an element of $Z^1(u_{\R} \rightarrow \mathcal{E}_{\R}, \mathbf{Z} \rightarrow \mathbf{T}^*)$.
Note that for now we do not assume that $\mathbf{G}$ is a pure inner form in the even orthogonal case, but we will have to make this assumption again in section \ref{sec:AJArthur}.

According to \cite[§5.6]{Kal}, which makes precise certain constants in Shelstad's work (\cite{She1}, \cite{She2}, \cite{She3}), this realisation and the choice of a Whittaker datum for $\mathbf{G}^*$ are enough to normalise the local Langlands correspondence for $\mathbf{G}$, that is to determine an internal parametrisation of each L-packet $\Pi_{\varphi}(\mathbf{G})$ of representations of $\mathbf{G}(\R)$ corresponding to a tempered Langlands parameter $W_{\R} \rightarrow {}^L \mathbf{G}$, in a manner analogous to the quasi-split case reviewed in section \ref{sec:Apackqs}.
In the anisotropic case, the L-packet $\Pi_{\varphi}(\mathbf{G})$ is a singleton.

In the global setting we will also need to consider non-tempered Arthur-Langlands parameters, i.e.\ Arthur-Langlands parameters $W_{\R} \times \mathrm{SL}_2(\C) \rightarrow {}^L \mathbf{G}$ whose restriction to $\mathrm{SL}_2(\C)$ is non-trivial.
For the main result of this paper (Theorem \ref{theo:main}) we will restrict to Arthur-Langlands parameters having algebraic regular infinitesimal character, i.e.\ those considered by Adams and Johnson in \cite{AdJo}.
Adams and Johnson worked with arbitrary real reductive groups (not necessarily quasi-split) and proved endoscopic character relations similar to \ref{eqn:endcharrelqs}, but only up to a multiplicative constant.
We will have to check that Kaletha's definitions allow to remove this ambiguity (Proposition \ref{prop:AJrigid}), similarly to what we did in \cite[§4.2.2]{Taidimtrace} for the quasi-split case.
We will see that in the anisotropic case, these Adams-Johnson packets are singletons.

Recently Arancibia, Moeglin and Renard \cite{AMR} proved that the packets of representations of $\mathbf{G}^*(\R)$ constructed by Adams and Johnson coincide with those constructed by Arthur in \cite{Arthur}, compatibly with the internal parametrisations.
We will show (Proposition \ref{prop:realAJArthur}) that this implies endoscopic character relations between Adams-Johnson packets for $\mathbf{G}$ and Arthur packets reviewed in section \ref{sec:Apackqs} for endoscopic groups of $\mathbf{G}$.
The only difficulty comes from the fact that Arthur packets are sets of representations \emph{up to outer automorphism}.
For this reason in the case $\mathbf{G}^* = \mathbf{SO}_{2n}^{\alpha}$ we will restrict ourselves to \emph{pure} inner forms $\mathbf{G}$.
Of course this covers the case of anisotropic inner forms in which we are particularly interested.

\subsubsection{Tempered parameters}

We begin by recalling Shelstad and Kaletha's internal parametrisation of L-packets for rigid inner twists of quasi-split groups \cite[§5.6]{Kal}.
We adopt a formulation slightly differing from \cite{Kal} and closer to \cite[§4.2.1]{Taidimtrace}, to prepare for the non-tempered case.
We will be brief and focus on the aspects which were not examined in \cite[§4.2.1]{Taidimtrace}.
It is convenient to let the inner form $\mathbf{G}$ of $\mathbf{G}^*$ vary.
Note that since $\mathbf{G}$ and $\mathbf{G}^*$ are inner forms of each other, their Langlands dual groups are canonically identified, and we will only use the notation ${}^L \mathbf{G}$.
Recall that by definition (\cite{BorelCorvallis}) $\widehat{\mathbf{G}}$ comes with a pinning $(\mathcal{B}, \mathcal{T}, (\mathcal{X}_{\alpha})_{\alpha \in \Delta})$, used to define the action of $W_{\R}$ on $\widehat{\mathbf{G}}$ and ${}^L \mathbf{G} = \widehat{\mathbf{G}} \rtimes W_{\R}$.
The $\widehat{\mathbf{G}}$-conjugacy classes of \emph{discrete} Langlands parameters $\varphi_{\lambda} : W_{\R} \rightarrow {}^L \mathbf{G}^*$ are parametrised by dominant weights $\lambda$ for $\mathbf{G}$.
We can assume that $\varphi_{\lambda}$ is aligned with $(\mathcal{B}, \mathcal{T})$ in the sense that $\varphi_{\lambda}(z) = (2 \tau)(z/|z|) \rtimes z$ for $z \in W_{\C} = \C^{\times}$, where $\tau = \lambda + \rho$ (with $2 \rho$ equal to the sum of the positive roots) is seen as an element of $\frac{1}{2} \Z \otimes X_*(\mathcal{T})$, strictly dominant for $\mathcal{B}$.
See \cite[Lemma 3.2]{Langlands} for the extension of $\varphi_{\lambda}$ to $W_{\R}$.
Since $\tau$ is strictly dominant, $C_{\varphi_{\lambda}} = \{ t \in \mathcal{T}\ |\ t^2 = 1 \}$.
Denote by $C_{\varphi_{\lambda}}^+$ the preimage of $C_{\varphi_{\lambda}}$ in $\widehat{\overline{\mathbf{G}}}$ where $\overline{\mathbf{G}} = \mathbf{G} / \mathbf{Z}$.
Denote by $Z(\widehat{\overline{\mathbf{G}}})^+$ the preimage of $Z(\widehat{\mathbf{G}})^{\mathrm{Gal}(\C / \R)}$ in $Z(\widehat{\overline{\mathbf{G}}})$.
Then $Z(\widehat{\overline{\mathbf{G}}})^+ \subset C_{\varphi_{\lambda}}^+$.
Let $\mathcal{S}_{\varphi_{\lambda}}^+ = \pi_0(C_{\varphi_{\lambda}}^+)$, so that there is a surjective morphism $\mathcal{S}_{\varphi_{\lambda}}^+ \rightarrow \mathcal{S}_{\varphi_{\lambda}}$ whose kernel is the image of $Z(\widehat{\overline{\mathbf{G}}})^+$ in $\mathcal{S}_{\varphi_{\lambda}}^+$.
We have that $C_{\varphi_{\lambda}}^+$ has dimension $0$ and so $\mathcal{S}_{\varphi_{\lambda}}^+ = C_{\varphi_{\lambda}}^+$ and the above kernel is $Z(\widehat{\overline{\mathbf{G}}})^+$.
Note also that $C_{\varphi_{\lambda}}^+$ is contained in a maximal torus of $\widehat{\overline{\mathbf{G}}}$, and thus $\mathcal{S}_{\varphi_{\lambda}}^+$ is abelian.

For any inner form $\mathbf{G}$ of $\mathbf{G}^*$, the L-packet $\Pi_{\varphi_{\lambda}}(\mathbf{G})$ of discrete series representations of $\mathbf{G}(\R)$ associated with $\varphi_{\lambda}$ is in bijection with the set $\Sigma_{\mathcal{B}}(\mathbf{G})$ of $\mathbf{G}(\R)$-conjugacy classes of pairs $(\mathbf{B}, \mathbf{T})$ where $\mathbf{T}$ is an anisotropic maximal torus of $\mathbf{G}$ and $\mathbf{B}$ is a Borel subgroup of $\mathbf{G}_{\C}$ containing $\mathbf{T}_{\C}$.
Choose a Whittaker datum $\mathfrak{w}$ for $\mathbf{G}^*$.
By work of Konstant and Vogan, for any dominant weight $\lambda$ there is a unique representation in $\Pi_{\varphi_{\lambda}}(\mathbf{G}^*)$ which is generic for $\mathfrak{w}$.
It corresponds to a $\mathbf{G}^*(\R)$-conjugacy class $\mathrm{cl}(\mathbf{B}^*, \mathbf{T}^*)$ which does not depend on $\lambda$ and has the property that all the simple roots of $\mathbf{T}^*_{\C}$ in $\mathbf{B}^*$ are non-compact.
There is a bijection between $H^1(u_{\R} \rightarrow \mathcal{E}_{\R}, \mathbf{Z} \rightarrow \mathbf{T}^*)$ and the set of isomorphism classes of $(\mathbf{G}, \Xi, z, \mathbf{B}, \mathbf{T})$ where $(\mathbf{G}, \Xi, z)$ is a rigid inner twist of $\mathbf{G}^*$ and $(\mathbf{B}, \mathbf{T})$ is a pair in $\mathbf{G}$ as above, obtained by mapping $z \in Z^1(u_{\R} \rightarrow \mathcal{E}_{\R}, \mathbf{Z} \rightarrow \mathbf{T}^*)$ to $(\mathbf{G}, \Xi, z, \Xi(\mathbf{B}^*), \Xi(\mathbf{T}^*))$, for an arbitrary choice of inner form $(\mathbf{G}, \Xi)$ compatible with $z$.
The surjectivity essentially comes from the fact that any automorphism of $\mathbf{T}^*_{\C}$ is defined over $\R$.
This gives a bijection between $H^1(u_{\R} \rightarrow \mathcal{E}_{\R}, \mathbf{Z} \rightarrow \mathbf{T}^*)$ and the set of isomorphism classes of pairs $((\mathbf{G}, \Xi, z), \pi)$ where $(\mathbf{G}, \Xi, z)$ is a rigid inner twist of $\mathbf{G}^*$ and $\pi \in \Pi_{\varphi_{\lambda}}(\mathbf{G})$.
Finally, the choice of $\mathbf{B}^*$ allows to identify $\mathcal{T}$ with $\widehat{\mathbf{T}^*}$, and thus to identify $\left(\mathcal{S}_{\varphi_{\lambda}}^+ \right)^{\vee}$ with $H^1(u_{\R} \rightarrow \mathcal{E}_{\R}, \mathbf{Z} \rightarrow \mathbf{T}^*)$, using \cite{Kal}[Corollary 5.4].
This concludes the construction of the mapping $\Pi_{\varphi_{\lambda}}(\mathbf{G}) \rightarrow \left(\mathcal{S}_{\varphi_{\lambda}}^+ \right)^{\vee}$, for any rigid inner twist $(\mathbf{G}, \Xi, z)$ of $\mathbf{G}^*$.
We will abusively denote this mapping by $\pi \mapsto \langle \cdot, \pi \rangle$, although it depends on the realisation of $\mathbf{G}$ as a rigid inner twist of $\mathbf{G}^*$ and the choice of the Whittaker datum $\mathfrak{w}$.
Shelstad and Kaletha also prove endoscopic character relations, analogous to equation \ref{eqn:endcharrelqs}.
Fix the dominant weight $\lambda$ and denote $\varphi = \varphi_{\lambda}$ for simplicity.
Any element $\dot{s} \in \mathcal{S}_{\varphi}^+$ yields a \emph{refined endoscopic datum}, a notion defined in \cite[§5.3]{Kal}, $\dot{\mathfrak{e}} = (\mathbf{H}, \mathcal{H}, \dot{s}, \xi)$ for $\mathbf{G}^*$, along with a parameter $\varphi' : W_{\R} \rightarrow \mathcal{H}$ such that $\varphi = \xi \circ \varphi'$.
Moreover the underlying endoscopic datum is elliptic.
Choose an L-embedding ${}^L \xi$ for this endoscopic datum, which is equivalent to choosing an isomorphism ${}^L \mathbf{H} \simeq \mathcal{H}$.
This allows us to see $\varphi'$ as a parameter for $\mathbf{H}$, and $\varphi'$ is clearly discrete.
For any rigid inner twist $(\mathbf{G}, \Xi, z)$ of $\mathbf{G}^*$, Kaletha defines absolute transfer factors $\Delta'[\dot{\mathfrak{e}}, {}^L \xi, \mathfrak{w}, \Xi, z]$.
We replaced the z-pair denoted by $\mathfrak{z}$ in \cite{Kal} by ${}^L \xi$, as we argued in section \ref{sec:ellenddat} that it is not necessary to introduce a z-extension of $\mathbf{H}$. 
By \cite[Proposition 5.10]{Kal}, if $f(g)dg$ and $f(h)dh$ are $\Delta'[\dot{\mathfrak{e}}, {}^L \xi, \mathfrak{w}, \Xi, z]$-matching smooth compactly supported distributions on $\mathbf{G}(\R)$ and $\mathbf{H}(\R)$, we have
\begin{equation} \label{eqn:endcharrelrealtemp}
e(\mathbf{G}) \sum_{\pi \in \Pi_{\varphi}(\mathbf{G})} \langle \dot{s}, \pi \rangle \mathrm{tr}\left(\pi(f(g)dg)\right) = \sum_{\pi' \in \Pi_{\varphi'}(\mathbf{H})} \mathrm{tr}\left(\pi'(f'(h)dh)\right).
\end{equation}
Here $e(\mathbf{G}) = (-1)^{q(\mathbf{G})-q(\mathbf{G}^*)}$ is the sign defined by Kottwitz in \cite{KottSign}.
\begin{rema}
The restriction of each $\langle \cdot, \pi \rangle$ to $Z(\widehat{\overline{\mathbf{G}}})^+$ coincides with the image of $z$ by \cite{Kal}[Corollary 5.4], which justifies the use of the transfer factors $\Delta'$ rather than $\Delta$ (see \cite{KS12}).
Note that in section \ref{sec:Apackqs} $\Delta$ was used, but we could have used $\Delta'$ as well since for any endoscopic datum considered in this paper we have $s^2=1$.
\end{rema}
\begin{rema} \label{rema:singletontemp}
For an \emph{anisotropic} rigid inner twist $(\mathbf{G}, \Xi, z)$ of $\mathbf{G}^*$, there is a unique $\mathbf{G}(\R)$-conjugacy class of pairs $(\mathbf{B}, \mathbf{T})$ in $\mathbf{G}$ as above.
Therefore for any discrete parameter $\varphi$ the L-packet $\Pi_{\varphi}(\mathbf{G})$ consists of only one representation. 
\end{rema}

\begin{exam}
Let us end this section by discussing an example.
For $n \geq 2$ let $\mathbf{G}^* = \mathbf{SO}_{2n}^{(-1)^n} = \mathbf{SO}(V, q)$ where $q$ has signature $(n,n)$ (resp.\ $(n+1,n-1)$) if $n$ is even (resp.\ odd).
We use the same notation as in section \ref{sec:innerformsreal}.
If $n$ is odd the quasi-split group $\mathbf{G}^*$ has a unique $\mathbf{G}^*(\R)$-conjugacy class of Whittaker datum $\mathfrak{w}$.
If $n$ is even there are two such conjugacy classes, and we let $\mathfrak{w}$ be the one corresponding to the pair $(\mathbf{B}^*, \mathbf{T}^*)$ chosen in section \ref{sec:innerformsreal}.
Let $(\mathbf{G}, \Xi, z)$ where $z \in Z^1(\R, \mathbf{T}^*)$ be the pure inner twist of $\mathbf{G}^*$ corresponding to the signature $(2n,0)$.

We realise the dual group ${}^L \mathbf{G} = {}^L \mathbf{G}^*$ as in section \ref{sec:defgroups}, using the symmetric bilinear form $B(v_i, v_j) = \delta_{i, 2n+1-j}$ on $\C v_1 \oplus \dots \oplus v_{2n}$.
Recall that the set of simple roots for $(\mathbf{B}^*, \mathbf{T}^*)$ is $\Delta = \{e_1-e_2, \dots, e_{n-1}-e_n, e_{n-1}+e_n\}$.
Let $\mathcal{T}$ be the maximal torus of $\widehat{\mathbf{G}}$ consisting of diagonal matrices $t = \mathrm{diag}(t_1, \dots, t_n, t_n^{-1}, \dots, t_1^{-1})$ for $t_1, \dots, t_n \in \C^{\times}$.
Let $\mathcal{B} \supset \mathcal{T}$ be the Borel subgroup of $\widehat{\mathbf{G}}$ consisting of upper triangular matrices in $\widehat{\mathbf{G}}$.
The corresponding set of simple roots is $\Delta^* = \{e_1^*-e_2^*, \dots, e_{n-1}^*-e_{n}^*, e_{n-1}^*+e_n^*\}$ where $e_j^*(t) = t_j$.
There exists a pinning $(X_{\alpha^*})_{\alpha^* \in \Delta^*}$ for $(\mathcal{B}, \mathcal{T})$ preserved by $\{1\} \rtimes \mathrm{Gal}(\C / \R) \subset {}^L \mathbf{G}$.
Thanks to the choice of Borel subgroups there is a canonical isomorphism $\widehat{\mathbf{T}^*} \simeq \mathcal{T}$, and we chose the notation for $\Delta^*$ accordingly.
The class of $z$ in $H^1(\R, \mathbf{T}^*)$ corresponds to the character $e_2^* + e_4^* + \dots + e_{2 \lfloor n/2 \rfloor}^* \in X^*(\mathcal{T}) / 2 X^*(\mathcal{T}) = X_*(\mathbf{T}^*) / 2 X_*(\mathbf{T}^*)$ of $\{ t \in \mathcal{T} \,|\, t^2 = 1 \}$.

Let $\lambda$ be a dominant weight for $\mathbf{G}^*$, i.e.\ $\lambda = \lambda_1 e_1 + \dots + \lambda_n e_n$ where $\lambda_1 \geq \dots \geq \lambda_{n-1} \geq |\lambda_n|$ are integers.
The corresponding discrete parameter $\varphi_{\lambda}$ can be chosen (in the canonical $\widehat{\mathbf{G}}$-conjugacy class) so that for all $z \in W_{\C}$, $\varphi_{\lambda}(z) = \mathrm{diag}((z/\bar{z})^{\lambda_1+n-1}, \dots, (z/\bar{z})^{\lambda_n}, \dots) \in \mathcal{T}$, i.e.\ $\varphi_{\lambda}|_{W_{\C}}$ is dominant for $(\mathcal{B}, \mathcal{T})$.
Thus the unique element of the L-packet $\Pi_{\varphi_{\lambda}}(\mathbf{G})$ corresponds to the character $e_2^* + e_4^* + \dots + e_{2 \lfloor n/2 \rfloor}^*$ of $\{ t \in \mathcal{T} \,|\, t^2 = 1 \} = C_{\varphi_{\lambda}}$.
If $n$ is even and if we choose the other Whittaker datum, the associated character becomes $e_1^* + \dots + e_{n-1}^*$.
\end{exam}

\subsubsection{Non-tempered parameters: Adams-Johnson packets}
\label{sec:AJ}

We now turn to the non-tempered case.
As before $\mathbf{G}$ will denote any inner form of $\mathbf{G}^*$.
Let $\psi : W_{\R} \times \mathrm{SL}_2(\C) \rightarrow {}^L \mathbf{G}$ be an Arthur-Langlands parameter which is bounded on $W_{\R}$.
We say that $\psi$ is an Adams-Johnson parameter if the central character of $\psi$ is algebraic regular.
This is the case if $\psi$ is trivial on $\mathrm{SL}_2(\C)$ and defines a discrete Langlands parameter, but there are other examples as well.
An Adams-Johnson parameter is always discrete, i.e.\ its centraliser in $\widehat{\mathbf{G}}$ is contained in $Z(\widehat{\mathbf{G}})$, but not all discrete Arthur-Langlands parameter are Adams-Johnson parameters.
From now on we assume that $\psi$ is an Adams-Johnson parameter.
We refer to \cite[§4.2.2]{Taidimtrace} for the details of what follows.
Recall that there is a Langlands parameter $\varphi_{\psi}$ associated with $\psi$, and that this parameter is not tempered if $\psi|_{\mathrm{SL}_2(\C)}$ is not trivial.
Up to conjugating by $\widehat{\mathbf{G}}$ we can assume that $\varphi_{\psi}|_{W_{\C}}$ takes values in $\mathcal{T}$ and that the holomorphic part of $\varphi_{\psi}|_{W_{\C}}$ coincides with that of $\varphi_{\lambda}|_{W_{\C}}$ as above, for a dominant weight $\lambda$ for $\mathbf{G}$.
Then, defining $C_{\psi}$, $C_{\psi}^+$ and $\mathcal{S}_{\psi}^+$ as in the tempered case, we have $C_{\psi} \subset C_{\varphi_{\lambda}}$, $C_{\psi}^+ \subset C_{\varphi_{\lambda}}^+$ and $\mathcal{S}_{\psi}^+ \subset \mathcal{S}_{\varphi_{\lambda}}^+$.
Adams and Johnson \cite{AdJo} have attached a packet $\Pi_{\psi}^{\mathrm{AJ}}(\mathbf{G})$ of irreducible representations of $\mathbf{G}(\R)$ to $\psi$, and proved endoscopic character relations similar to \ref{eqn:endcharrelqs}, but only up to a multiplicative constant because only \emph{relative} transfer factors were known at the time.
Exactly as in \cite[§4.2.2]{Taidimtrace} which only adressed the quasi-split case, we will argue that knowing the tempered case above is enough to remove this ambiguity.
Let $\mathcal{L} = \mathrm{Cent}(\psi(W_{\C}), \widehat{\mathbf{G}})$.
Similarly to the tempered case, $\Pi_{\psi}^{\mathrm{AJ}}(\mathbf{G})$ is in bijection with a certain set $\Sigma_{\mathcal{L}}(\mathbf{G})$ of $\mathbf{G}(\R)$-conjugacy classes of pairs $(\mathbf{Q}, \mathbf{L})$, and there is a natural surjection $\Sigma_{\mathcal{B}}(\mathbf{G}) \rightarrow \Sigma_{\mathcal{L}}(\mathbf{G})$.
The choice of a Whittaker datum gives us a distinguished element $\mathrm{cl}(\mathbf{Q}^*, \mathbf{L}^*) \in \Sigma_{\mathcal{L}}(\mathbf{G}^*)$, that is the image of $\mathrm{cl}(\mathbf{B}^*, \mathbf{T}^*) \in \Sigma_{\mathcal{B}}(\mathbf{G}^*)$.
There is a natural bijection between $H^1(u_{\R} \rightarrow \mathcal{E}_{\R}, \mathbf{Z} \rightarrow \mathbf{L}^*)$ and the set of isomorphism classes of pairs $((\mathbf{G}, \Xi, z), (\mathbf{Q}, \mathbf{L}))$ where $(\mathbf{G}, \Xi, z)$ is a rigid inner twist of $\mathbf{G}^*$ and $\mathrm{cl}(\mathbf{Q}, \mathbf{L}) \in \Sigma_{\mathcal{L}}(\mathbf{G})$.
This realises each such $\mathbf{L}$ as a rigid inner twist of $\mathbf{L}^*$.
Now fix a rigid inner twist $(\mathbf{G}, \Xi, z)$ of $\mathbf{G}^*$.
By \cite[§5]{ArthurUnip}, if $\mathrm{cl}(\mathbf{B}_1, \mathbf{T}_1), \mathrm{cl}(\mathbf{B}_2, \mathbf{T}_2) \in \Sigma_{\mathcal{B}}(\mathbf{G})$ map to the same element in $\Sigma_{\mathcal{L}}(\mathbf{G})$, then their associated characters of $\mathcal{S}_{\varphi_{\lambda}}^+$ coincide on $\mathcal{S}_{\psi}^+$.
Thus we get a mapping $\Pi_{\psi}^{\mathrm{AJ}}(\mathbf{G}) \rightarrow \left(\mathcal{S}_{\varphi_{\lambda}}^+\right)^{\vee}$, which as in the tempered case we simply denote by $\pi \mapsto \langle \cdot, \pi \rangle$.
\begin{rema} \label{rema:singletonAJ}
As in the tempered case and Remark \ref{rema:singletontemp}, in the case of an \emph{anisotropic} rigid inner twist $(\mathbf{G}, \Xi, z)$, for any Adams-Johnson parameter $\psi$ the packet $\Pi_{\psi}^{\mathrm{AJ}}(\mathbf{G})$ is a singleton.
\end{rema}
Observe that the preimage $\widetilde{\mathcal{L}}$ of $\mathcal{L}$ in $\widehat{\overline{\mathbf{G}}}$ is isomorphic to $\widehat{\overline{\mathbf{L}}} = \widehat{\overline{\mathbf{L}^*}}$, and that the $\mathrm{Gal}(\C / \R)$ actions coincide if we let $\mathrm{Gal}(\C / \R)$ act on $\widetilde{\mathcal{L}}$ via $\psi|_{W_{\R}}$.
The pairing between $H^1(u_{\R} \rightarrow \mathcal{E}_{\R}, \mathbf{Z} \rightarrow \mathbf{L}^*)$ and $\mathcal{S}_{\psi}^+ \simeq Z(\widehat{\overline{\mathbf{L}}})^+$ is naturally the one given by \cite{Kal}[Corollary 5.4].
Finally we observe that there is a canonical lift of $s_{\psi} \in C_{\psi}$ to $C_{\psi}^+$, since $\mathrm{SL}_2(\C)$ is simply connected.
In fact $s_{\psi}$ is well-defined in $Z(\mathcal{L}_{\mathrm{sc}})^{\mathrm{Gal}(\C / \R)} = Z(\widehat{\mathbf{L}_{\mathrm{ad}}})^{\mathrm{Gal}(\C / \R)}$: $\psi|_{\mathrm{SL}_2(\C)} : \mathrm{SL}_2(\C) \rightarrow \mathcal{L}$ is the principal morphism and thus $s_{\psi}$ is equal to the sum of the positive coroots for $\mathcal{L}$ evaluated at $-1$, for any choice of ordering on the coroots.
We still denote its image in $C_{\psi}^+$ and $\mathcal{S}_{\psi}^+$ by $s_{\psi}$.
We can now prove the precise endoscopic character relations, refining \cite[Theorem 2.21]{AdJo}.
\begin{prop} \label{prop:AJrigid}
Let $\dot{s} \in \mathcal{S}_{\psi}^+$, which induces a refined endoscopic datum $\dot{\mathfrak{e}} = (\mathbf{H}, \mathcal{H}, \dot{s}, \xi)$ for $\mathbf{G}$, and choose an L-embedding ${}^L \xi : {}^L \mathbf{H} \rightarrow {}^L \mathbf{G}$.
Denote by $\psi'$ the Adams-Johnson parameter for $\mathbf{H}$ such that $\psi = {}^L \xi \circ \psi'$.
For any $\Delta'[\dot{\mathfrak{e}}, {}^L \xi, \mathfrak{w}, \Xi, z]$-matching smooth compactly supported distributions $f(g)dg$ on $\mathbf{G}(\R)$ and $f'(h)dh$ on $\mathbf{H}(\R)$, we have
\begin{equation} \label{eqn:endcharrelrealAJ}
e(\mathbf{G}) \sum_{\pi \in \Pi_{\psi}^{\mathrm{AJ}}(\mathbf{G})} \langle \dot{s} s_{\psi}, \pi \rangle \mathrm{tr}\left(\pi(f(g)dg)\right) = \sum_{\pi' \in \Pi_{\psi'}^{\mathrm{AJ}}(\mathbf{H})} \langle s_{\psi'}, \pi' \rangle \mathrm{tr}\left(\pi'(f'(h)dh)\right).
\end{equation}
\end{prop}
\begin{proof}
Denote by $\pi_{\psi, \mathbf{Q}, \mathbf{L}} \in \Pi_{\psi}^{\mathrm{AJ}}(\mathbf{G})$ the representations associated with $\mathrm{cl}(\mathbf{Q}, \mathbf{L}) \in \Sigma_{\mathbf{L}}(\mathbf{G})$.
As usual denote $q(\mathbf{L}) = \frac{1}{2}\dim \mathbf{L}(\R) / K_{\mathbf{L}}$ for any maximal compact subgroup $K_{\mathbf{L}}$ of $\mathbf{L}(\R)$.
Recall Kottwitz' sign $e(\mathbf{L}) = (-1)^{q(\mathbf{L}) - q(\mathbf{L}_{\mathrm{qs}})}$ where $\mathbf{L}_{\mathrm{qs}}$ is the quasi-split inner form of $\mathbf{L}$.
We need a preliminary result similar to \cite[Lemma 9.1]{KottAA} or \cite[Lemma 5.1]{ArthurUnip}, namely that $\langle s_{\psi}, \pi \rangle = e(\mathbf{L})$.
By \cite[Lemma 4.2.1]{Taidimtrace} the group $\mathbf{L}^*$ is quasi-split, i.e.\ $\mathbf{L}_{\mathrm{qs}} = \mathbf{L}^*$.
The representation $\pi$ corresponds to an element of $ H^1(u_{\R} \rightarrow \mathcal{E}_{\R}, \mathbf{Z} \rightarrow \mathbf{L}^*)$, and $\langle s_{\psi}, \pi \rangle$ is obtained using the pairing \cite{Kal}.
This pairing is compatible with the pairing between $H^1(\R, \mathbf{L}^*_{\mathrm{ad}})$ and $Z(\widehat{\mathbf{L}}_{\mathrm{sc}})^{\mathrm{Gal}(\C / \R)}$ and $s_{\psi}$ is well-defined in the latter, as the sum of the positive coroots evaluated at $-1$.
It is well-known that this recovers $e(\mathbf{L})$.

Reformulating \cite[Theorem 2.21]{AdJo} using this preliminary result, we get that there exists $c \in \C^{\times}$ such that
$$ e(\mathbf{G}) \sum_{\pi \in \Pi_{\psi}^{\mathrm{AJ}}(\mathbf{G})} \langle \dot{s} s_{\psi}, \pi \rangle \mathrm{tr}\left(\pi(f(g)dg)\right) = c \sum_{\pi' \in \Pi_{\psi'}^{\mathrm{AJ}}(\mathbf{H})} \langle s_{\psi'}, \pi' \rangle \mathrm{tr}\left(\pi'(f'(h)dh)\right), $$
and we are left to prove that $c=1$.
We proceed as in the quasi-split case \cite[§4.2.2]{Taidimtrace}: we restrict to distributions $f(g)dg$ whose support is contained in semisimple elliptic regular conjugacy classes of $\mathbf{G}(\R)$.
Let $\varphi = \varphi_{\lambda}$ be the discrete Langlands parameter aligned with $\psi$ as above.
Then using Johnson's resolution of Adams-Johnson representations \cite{Johnson}, 
$$ \sum_{\pi \in \Pi_{\psi}^{\mathrm{AJ}}(\mathbf{G})} \langle \dot{s} s_{\psi}, \pi \rangle \mathrm{tr}\left(\pi(f(g)dg)\right) = (-1)^{q(\mathbf{L}^*)} \sum_{\pi \in \Pi_{\varphi}(\mathbf{G})} \langle \dot{s}, \pi \rangle \mathrm{tr}\left(\pi(f(g)dg)\right) $$
and similarly for $\mathbf{H}$.
Note that the sign $(-1)^{q(\mathbf{L}^*)}$ is identical for $\psi$ and $\psi'$, since the groups $\mathcal{L} \subset \widehat{\mathbf{G}}$ and $\mathcal{L}' \subset \widehat{\mathbf{H}}$ are naturally isomorphic, compatibly with their Galois actions.
As a consequence of \ref{eqn:endcharrelrealtemp} and these equalities for $\psi$ and $\psi'$, the relation \ref{eqn:endcharrelrealAJ} holds for such distributions $f(g)dg$.
There exists such a distribution for which the left hand side of \ref{eqn:endcharrelrealAJ} does not vanish, and thus $c=1$ and \ref{eqn:endcharrelrealAJ} is valid for any distribution.
\end{proof}

\subsubsection{Relation with Arthur packets}
\label{sec:AJArthur}

For our global purposes, which will rely on \cite{Arthur}, we would like to replace $\Pi_{\psi'}^{\mathrm{AJ}}(\mathbf{H})$ in the right hand side of equation \ref{eqn:endcharrelrealAJ} by Arthur's packet $\Pi_{\psi'}(\mathbf{H})$ reviewed in section \ref{sec:Apackqs}, that is
\begin{equation} \label{eqn:endcharrelrealAJArthur}
e(\mathbf{G}) \sum_{\pi \in \Pi_{\psi}^{\mathrm{AJ}}(\mathbf{G})} \langle \dot{s} s_{\psi}, \pi \rangle \mathrm{tr}\left(\pi(f(g)dg)\right) = \sum_{\pi' \in \Pi_{\psi'}(\mathbf{H})} \langle s_{\psi'}, \pi' \rangle \mathrm{tr}\left(\pi'(f'(h)dh)\right).
\end{equation}
For this to even make sense it is necessary for all the stable orbital integrals of $f'(h)dh$ to be invariant under the subgroup $\mathrm{Out}(\mathbf{H}_1) \times \mathrm{Out}(\mathbf{H}_2)$ of $\mathrm{Out}(\mathbf{H})$, where $\mathbf{H} \simeq \mathbf{H}_1 \times \mathbf{H}_2$ as in section \ref{sec:ellenddat}.

In the quasi-split case, more precisely for the trivial rigid inner twist of $\mathbf{G}^*$, it is known (and was used by Arthur) that this is the case if $f(g)dg$ has $\mathrm{Out}(\mathbf{G})$-invariant orbital integrals, for the choice of rational outer automorphism described in section \ref{sec:Apackqs}.
In this case Arancibia, Moeglin and Renard \cite{AMR} recently proved that for any Adams-Johnson parameter $\psi$, the packets of Adams-Johnson and Arthur and the maps to $\left(\mathcal{S}_{\psi}\right)^{\vee}$ coincide: this is equivalent to equation \ref{eqn:endcharrelrealAJArthur} for all $\psi$ but with $\dot{s}=1$, by Fourier inversion on $\mathcal{S}_{\psi}$.

Consequently for a non-trivial rigid inner twist $(\mathbf{G}, \Xi, z)$ equation \ref{eqn:endcharrelrealAJArthur} holds provided that all stable orbital integrals of $f'(h)dh$ are invariant under $\mathrm{Out}(\mathbf{H}_1) \times \mathrm{Out}(\mathbf{H}_2)$.
As explained in Remark \ref{rema:onlypureevenorth}, in the even orthogonal case to even formulate the $\mathrm{Out}(\mathbf{G})$-invariance of orbital integrals of $f(g)dg$, we have to restrict to groups $\mathbf{G}$ realised as \emph{pure} inner twists $(\mathbf{G}, \Xi, z)$, i.e.\ assume that $z \in Z^1(\R, \mathbf{G}^*)$.
We make this assumption and treat the three types of groups separately.
As before we denote by $\dot{\mathfrak{e}}$ the refined endoscopic datum $(\mathbf{H}, \mathcal{H}, \dot{s}, \xi)$ and by $\mathfrak{e}$ the underlying endoscopic datum, and fix an L-embedding ${}^L \xi : {}^L \mathbf{H} \rightarrow {}^L \mathbf{G}$.
\begin{itemize}
\item If $\mathbf{G}^* = \mathbf{SO}_{2n+1}$, the group $\mathrm{Out}(\mathbf{H}_1) \times \mathrm{Out}(\mathbf{H}_2)$ is always trivial and so there is nothing to prove.
\item If $\mathbf{G}^* = \mathbf{Sp}_{2n}$ and $\mathbf{H} \simeq \mathbf{Sp}_{2a} \times \mathbf{SO}_{2b}^{\alpha}$ with $b>0$, then $\mathrm{Out}(\mathbf{H}_1)$ is trivial, $\mathrm{Out}(\mathbf{H}_2) \simeq \Z / 2 \Z$ and
$$ \mathrm{Out}(\mathbf{H}_1) \times \mathrm{Out}(\mathbf{H}_2) = \mathrm{Out}(\mathbf{H}) \simeq \mathrm{Out}(\mathfrak{e}) = \mathrm{Out}(\dot{\mathfrak{e}}). $$
The last equality is crucial: in general we only have $\mathrm{Out}(\dot{\mathfrak{e}}) \subset \mathrm{Out}(\mathfrak{e})$, while the transfer factors defined in \cite[§5.3]{Kal} are only invariant by isomorphisms of \emph{refined} endoscopic data \cite[Proposition 5.6]{Kal}.
Furthermore, the L-embedding ${}^L \xi$ is preserved by $\mathrm{Out}(\dot{\mathfrak{e}})$.
Therefore, if $\mathrm{cl}(\gamma_1)$ and $\mathrm{cl}(\gamma_2)$ are stable conjugacy classes of semisimple regular elements of $\mathbf{H}(\R)$ exchanged by the outer automorphism of $\mathbf{H}$, then $\gamma_1$ is a norm of a conjugacy class $\mathrm{cl}(\delta)$ in $\mathbf{G}(\R)$ (in the sense of \cite[§3]{KS}) if and only if $\gamma_2$ is a norm of $\mathrm{cl}(\delta)$, and if this is the case then
$$ \Delta'[\dot{\mathfrak{e}}, {}^L \xi, \mathfrak{w}, \Xi, z](\gamma_1, \delta) = \Delta'[\dot{\mathfrak{e}}, {}^L \xi, \mathfrak{w}, \Xi, z](\gamma_2, \delta). $$
This implies that for any smooth compactly supported distribution $f(g)dg$ on $\mathbf{G}(\R)$, the stable orbital integrals of any smooth compactly supported distribution $f'(h)dh$ on $\mathbf{H}(\R)$ which is a transfer of $f(g)dg$ are invariant under $\mathrm{Out}(\mathbf{H})$.
\item The last case to consider is $\mathbf{G}^* = \mathbf{SO}_{2n}^{\alpha}$ and $\mathbf{H} \simeq \mathbf{SO}_{2a}^{\beta} \times \mathbf{SO}_{2b}^{\gamma}$.
Using the same arguments as in the previous case, one could show that if $a,b>0$ then the transfer factors are invariant under $(\theta_1, \theta_2)$ where $\theta_i$ is the non-trivial outer automorphism of $\mathbf{H}_i$.
We still have to consider the remaining coset in $\mathrm{Out}(\mathbf{H}_1) \times \mathrm{Out}(\mathbf{H}_2)$, in parallel with the non-trivial outer automorphism $\theta$ of $\mathbf{G}$.
Since we only consider groups $\mathbf{G}$ realised as \emph{pure} inner twists $(\mathbf{G}, \Xi, z)$ we can use Waldspurger's explicit formulae \cite{WaldFormulaire} for the transfer factors.
One can lift $\theta$ to an automorphism of $\mathbf{G}$ defined over $\R$ in a natural way since $\mathbf{G}$ is naturally a special orthogonal group, by considering any $\R$-point of the non-trivial connected component of the corresponding orthogonal group.
We still denote by $\theta$ any such lift, it is unique up to composing it with $\mathrm{Ad}(g)$ for some $g \in \mathbf{G}(\R)$.
One can simply observe on Waldspurger's formulae that the transfer factors are invariant under the action of $(\theta_1, \theta_2)$ on $\mathbf{H}$ if $a,b>0$, as well as under the simultaneous action of $\theta$ on $\mathbf{G}$ and the non-trivial element of $\left(\mathrm{Out}(\mathbf{H}_1) \times \mathrm{Out}(\mathbf{H}_2)\right)/\{ 1, (\theta_1, \theta_2) \}$ (resp.\ $\mathrm{Out}(\mathbf{H})$) if $a,b>0$ (resp.\ if $ab=0$).
To be precise we ought to remark that Waldspurger leaves out the $\epsilon$-factor of \cite[§5.3]{KS}, but this is just a constant that does not change the invariance properties, and that Waldspurger computed $\Delta$ whereas we use $\Delta'$ in this paper.
This is not a real issue since the factors $\Delta'$, defined in \cite{KS12} and \cite{Kal}, are equal to the factors $\Delta$ for the refined endoscopic datum obtained by replacing $\dot{s}$ by $\dot{s}^{-1}$.

This implies that for any smooth compactly supported distribution $f(g)dg$ on $\mathbf{G}(\R)$ having $\theta$-invariant orbital integrals, the stable orbital integrals of any smooth compactly supported distribution $f'(h)dh$ on $\mathbf{H}(\R)$ which is a transfer of $f(g)dg$ are invariant under $\mathrm{Out}(\mathbf{H}_1) \times \mathrm{Out}(\mathbf{H}_2)$.
\end{itemize}
For clarity, let us restate what we have deduced from \cite{AMR}.
\begin{prop} \label{prop:realAJArthur}
Let $(\mathbf{G}, \Xi, z)$ be a rigid inner twist of a reductive group $\mathbf{G}^*$ over $\R$ as in section \ref{sec:defgroups}.
Assume that $\mathbf{G}^*$ admits discrete series.
If $\mathbf{G}^*$ is even orthogonal, assume also that $z \in Z^1(\R, \mathbf{G}^*)$ and let $\theta$ be an automorphism of $\mathbf{G}$ as in the above discussion.
Let $\psi$ be an Adams-Johnson parameter for $\mathbf{G}$.
Let $\dot{s} \in C_{\psi}^+$, let $\mathfrak{e} = (\mathbf{H}, \mathcal{H}, \dot{s}, \xi)$ be the induced elliptic refined endoscopic datum, and choose an L-embedding ${}^L \xi$ for this endoscopic datum.
As explained in section \ref{sec:ellenddat} $\mathbf{H}$ decomposes as a product $\mathbf{H}_1 \times \mathbf{H}_2$.
Let $f(g)dg$ be a smooth compactly supported distribution on $\mathbf{G}(\R)$.
In the case $\mathbf{G}^* = \mathbf{SO}_{2n}^{\alpha}$ we assume that the orbital integrals of $f(g)dg$ are $\theta$-invariant.
Let $f'(h)dh$ be a transfer of $f(g)dg$ with respect to the transfer factors $\Delta'[\dot{\mathfrak{e}}, {}^L \xi, \mathfrak{w}, \Xi, z]$.

Then the stable orbital integrals of $f'(h)dh$ are invariant under $\mathrm{Out}(\mathbf{H}_1) \times \mathrm{Out}(\mathbf{H}_2)$ and the endoscopic character relation \ref{eqn:endcharrelrealAJArthur} holds.
\end{prop}

\subsection{Trivial inner twists}
\label{sec:trivit}

In this section we denote by $F$ a local field of characteristic zero, and consider a quasi-split reductive group $\mathbf{G}^*$ over $F$ as in section \ref{sec:defgroups}, and a finite subgroup $\mathbf{Z}$ of the center of $\mathbf{G}^*$.
Consider also a rigid inner twist $(\mathbf{G}, \Xi, z)$ of $\mathbf{G}^*$.
Denote by $\overline{\mathbf{G}}^*$ and $\overline{\mathbf{G}}$ the quotients of $\mathbf{G}^*$ and $\mathbf{G}$ by $\mathbf{Z}$.
Assume that the image of $z$ in $H^1(F, \mathbf{G}^*_{\mathrm{ad}})$ is trivial.
For example this is what we obtain by localisation at a place $v \not\in S$ of a global rigid inner twist as in section \ref{sec:gri}.
Fix a Whittaker datum $\mathfrak{w}$ for $\mathbf{G}^*$.

If the class of $z$ in $H^1(u_F \rightarrow \mathcal{E}_F, \mathbf{Z} \rightarrow \mathbf{G}^*)$ is trivial, then $(\mathbf{G}, \Xi, z)$ is isomorphic to the trivial rigid inner twist $(\mathbf{G}^*, \mathrm{Id}, 1)$ and the isomorphism is unique up to $\mathbf{G}^*(F)$.
In particular $\mathbf{G}$ inherits a well-defined Whittaker datum $\mathfrak{w}'$ from $\mathfrak{w}$, and the absolute transfer factors $\Delta'[\dot{\mathfrak{e}}, {}^L \xi, \mathfrak{w}, \Xi, z]$ are equal to the absolute transfer factors $\Delta'[\mathfrak{e}, {}^L \xi, \mathfrak{w}']$ defined in \cite[§5.3]{KS}.
Therefore we can simply use the Arthur packets reviewed in section \ref{sec:Apackqs} for $\mathbf{G}$ and $\mathfrak{w}'$.

If however the class of $z$ in $H^1(u_F \rightarrow \mathcal{E}_F, \mathbf{Z} \rightarrow \mathbf{G}^*)$ is non-trivial, as can happen by Remark \ref{rema:probtriv}, then an isomorphism between $\mathbf{G}$ and $\mathbf{G}^*$ is only well-defined up to $\mathbf{G}^*_{\mathrm{ad}}(F)$, which does not preserve the $\mathbf{G}^*(F)$-orbit of $\mathfrak{w}$ in general.
Thus there is no natural choice of Whittaker datum for $\mathbf{G}$, and even if there was the corresponding absolute transfer factors would have no reason to coincide with $\Delta'[\dot{\mathfrak{e}}, {}^L \xi, \mathfrak{w}, \Xi, z]$.
We conclude that in this situation we cannot merely use the packets reviewed in section \ref{sec:Apackqs} along with the characters $\langle \cdot, \pi \rangle$ to get endoscopic character relations similar to equation \ref{eqn:endcharrelqs} but for the transfer factors $\Delta'[\dot{\mathfrak{e}}, {}^L \xi, \mathfrak{w}, \Xi, z]$.

Kaletha \cite[§6.3]{Kalrivsbg} computed the effect of twisting the cocycle $z \in Z^1(u_F \rightarrow \mathcal{E}_F, \mathbf{Z} \rightarrow \mathbf{G}^*)$ by an element of $Z^1(u_F \rightarrow \mathcal{E}_F, \mathbf{Z} \rightarrow \mathbf{Z}(\mathbf{G}^*))$ on the transfer factors $\Delta'[\dot{\mathfrak{e}}, {}^L \xi, \mathfrak{w}, \Xi, z]$ \cite[Lemma 6.3]{Kalrivsbg}, and deduced the effect on the spectral transfer factors $\langle \dot{s}, \pi \rangle$ \cite[Lemma 6.2]{Kalrivsbg}.
Let us recall how this applies to the present situation.
Up to replacing $(\mathbf{G}, \Xi, z)$ by an isomorphic rigid inner twist, we can assume that the image of $z$ in $Z^1(F, \mathbf{G}^*_{\mathrm{ad}})$ is trivial, i.e.\ that $z \in Z^1(u_F \rightarrow \mathcal{E}_F, \mathbf{Z} \rightarrow \mathbf{Z}(\mathbf{G}^*))$.
In particular, the isomorphism $\Xi : \mathbf{G}^*_{\overline{F}} \rightarrow \mathbf{G}_{\overline{F}}$ is defined over $F$.
Up to enlarging $\mathbf{Z}$ (as a finite subgroup of $\mathbf{Z}(\mathbf{G}^*)$), we can assume that $z \in Z^1(u_F \rightarrow \mathcal{E}_F, \mathbf{Z} \rightarrow \mathbf{Z})$.
The exact sequence $1 \rightarrow \mathbf{Z} \rightarrow \mathbf{G}^* \rightarrow \overline{\mathbf{G}}^* \rightarrow 1$ dualises as $1 \rightarrow \widehat{\mathbf{Z}} \rightarrow \widehat{\overline{\mathbf{G}}} \rightarrow \widehat{\mathbf{G}} \rightarrow 1$, where $\widehat{\mathbf{Z}}$ is naturally identified with $2 i \pi \mathbb{Z} \otimes X^*(\mathbf{Z})$.
Let $\psi : \mathrm{WD}_F \times \mathrm{SL}_2(\C) \rightarrow {}^L \mathbf{G}$ be an Arthur-Langlands parameter.
As in the real case, let $C_{\psi} = \mathrm{Cent}(\psi, \widehat{\mathbf{G}})$, $C_{\psi}^+$ its preimage in $\widehat{\overline{\mathbf{G}}}$ and $\mathcal{S}_{\psi}^+ = \pi_0(C_{\psi}^+)$.
Kaletha constructs a morphism $d : \mathcal{S}_{\psi}^+ \rightarrow Z^1(\mathrm{Gal}_F, \widehat{\mathbf{Z}})$.
By \cite[§6.2]{Kalrivsbg} there is a perfect pairing
\begin{equation} \label{eqn:pairingZ} Z^1(\mathrm{Gal}_F, \widehat{\mathbf{Z}}) \times H^1(u_F \rightarrow \mathcal{E}_F, \mathbf{Z} \rightarrow \mathbf{Z}) \rightarrow \C^{\times} \end{equation}
which we simply denote by $\langle \cdot, \cdot \rangle_{\mathbf{Z}}$.
These are the two constructions needed to state the existence of Arthur packets for $\mathbf{G}$.
To avoid confusion we will denote by $\Pi_{\psi}(\mathbf{G}^*)$ and $\theta^*$ what was denoted by $\Pi_{\psi}$ and $\theta$ in section \ref{sec:Apackqs}.
If $\mathbf{G}^* = \mathbf{SO}_{2n}^{\alpha}$, let $\theta = \Xi \circ \theta^* \circ \Xi^{-1}$.
Define a multiset $\Pi_{\psi}(\mathbf{G})$ of representations of $\mathbf{G}(F)$, or $\theta$-orbits of representations in the even orthogonal case, as $(\Xi^{-1})^*(\Pi_{\psi}(\mathbf{G}^*))$, and for $\pi \in \Pi_{\psi}(\mathbf{G})$ define a character $\langle \cdot, \pi \rangle_{\mathbf{G}}$ of $\mathcal{S}_{\psi}^+$ by
$$ \langle \dot{s}, \pi \rangle = \langle d (\dot{s})^{-1}, z \rangle_{\mathbf{Z}} \times \langle s, \pi \circ \Xi \rangle $$
for $\dot{s} \in \mathcal{S}_{\psi}^+$ and $s$ its image in $\mathcal{S}_{\psi}$.

Fix $\dot{s} \in C_{\psi}^+$, let $\mathfrak{e} = (\mathbf{H}, \mathcal{H}, \dot{s}, \xi)$ be the induced refined endoscopic datum, and choose an L-embedding ${}^L \xi : {}^L \mathbf{H} \rightarrow {}^L \mathbf{G}$ for this endoscopic datum.
As explained in section \ref{sec:ellenddat}, $\mathbf{H}$ decomposes as a product $\mathbf{H}_1 \times \mathbf{H}_2$.

\begin{prop} \label{prop:trivit}
Let $f(g)dg$ be a smooth compactly supported distribution on $\mathbf{G}(F)$.
In the even orthogonal case assume that the orbital integrals of $f(g)dg$ are invariant under $\theta$.
Let $f'(h)dh$ be a transfer of $f(g)dg$ with respect to the transfer factors $\Delta'[\dot{\mathfrak{e}}, {}^L \xi, \mathfrak{w}, \Xi, z]$.

Then the stable orbital integrals of $f'(h)dh$ are invariant under $\mathrm{Out}(\mathbf{H}_1) \times \mathrm{Out}(\mathbf{H}_2)$ and we have
\begin{equation} \sum_{\pi \in \Pi_{\psi}(\mathbf{G})} \langle \dot{s} s_{\psi}, \pi \rangle \mathrm{tr}\left(\pi(f(g)dg)\right) = \sum_{\pi' \in \Pi_{\psi'}(\mathbf{H})} \langle s_{\psi'}, \pi' \rangle \mathrm{tr}\left(\pi'(f'(h)dh)\right) \end{equation}
\end{prop}
\begin{proof}
This is a special case of \cite[Lemma 6.2]{Kalrivsbg}, except for the invariance under $\mathrm{Out}(\mathbf{H}_1) \times \mathrm{Out}(\mathbf{H}_2)$ of the stable orbital integrals of $f'(h)dh$.
The corresponding invariance property for the transfer factors $\Delta'[\dot{\mathfrak{e}}, {}^L \xi, \mathfrak{w}, \Xi, z]$ can be proved as in Proposition \ref{prop:realAJArthur}.
More simply, it can be deduced from the similar invariance property for the transfer factors for $\mathbf{G}^*$, since it is a property of the \emph{relative} transfer factors.
\end{proof}

\begin{rema} \label{rema:trivit}
By Remark \ref{rema:probtriv} in the global setting of section \ref{sec:gri} the only case where we cannot reduce to $z=1$ at any finite place is when $\mathbf{G}^* = \mathbf{SO}_{2n}^{\alpha}$.
In this case the restriction of $z \in Z^1(u_F \rightarrow \mathcal{E}_F, \mathbf{Z} \rightarrow \mathbf{Z})$ to $u_F$ is trivial, that is $z \in Z^1(F, \mathbf{Z})$.
Instead of the pairing \ref{eqn:pairingZ} of \cite[§§6.1-2]{Kalrivsbg} one can use the usual Poitou-Tate pairing
$$ H^1(\mathrm{Gal}_F, \widehat{\mathbf{Z}}) \times H^1(F, \mathbf{Z}) \rightarrow \C^{\times} $$
and the characters $\langle \dot{s}, \pi \rangle$ defined above for $\pi \in \Pi_{\psi}(\mathbf{G})$ descend to characters of $\pi_0(C_{\psi})$, which is always a $2$-group, in particular it is abelian.
\end{rema}

\subsection{Adèlic Arthur packets for rigid inner forms}
\label{sec:adelicApack}

Let us go back to the setting of section \ref{sec:gri}: $F$ is a totally real number field, $\mathbf{G}^*$ one of the quasi-split reductive groups over $F$ defined in section \ref{sec:defgroups}.
Let $\mathbf{Z}$ be the center of $\mathbf{G}^*$ if $\mathbf{G}^*$ is symplectic.
If $\mathbf{G}^*$ is orthogonal we let $\mathbf{Z}=1$.
Let $S$ be a non-empty set of real places of $F$, and consider a rigid inner twist $(\mathbf{G}, \Xi, z)$ of $\mathbf{G}^*$ satisfying property $\star$ with respect to $S$, where $z \in Z^1(P_{F, \dot{V}} \rightarrow \mathcal{E}_{F, \dot{V}}, \mathbf{Z} \rightarrow \mathbf{G}^*)$.
Such rigid inner twists were classified in Proposition \ref{prop:gri}.
Note that if $\mathbf{G}^*$ is orthogonal $z$ is simply an element of $Z^1(F, \mathbf{G}^*)$.
Fix also a global Whittaker datum $\mathfrak{w}$ for $\mathbf{G}^*$.

For any place $v$ of $F$ we get a rigid inner twist $(\mathbf{G}_{F_v}, \Xi_v, z_v)$ of $\mathbf{G}^*_{F_v}$, where
$$ z_v \in \begin{cases} Z^1(F_v, \mathbf{G}^*_{F_v}) & \text{ if } \mathbf{G}^* \text{ is orthogonal or } v \not\in S, \\ Z^1(u_v \rightarrow \mathcal{E}_{F_v}, \mathbf{Z} \rightarrow \mathbf{G}^*_{F_v}) & \text{ if } \mathbf{G}^* \text{ is symplectic and } v \in S, \end{cases} $$
is the localisation of $z$ at $v$ and $\Xi_v = \Xi_{\overline{F}_v}$.
We also get a Whittaker datum $\mathfrak{w}_v$ for $\mathbf{G}^*_{F_v}$.

Recall that $\mathbf{G}^*$ (resp.\ $\mathbf{G}$) admits a model over $\mathcal{O}_F[\frac{1}{N}]$ for some $N \geq 1$, and that up to replacing $N$ by a multiple two such models are canonically isomorphic.
Furthermore, there exists $N$ such that $\Xi$ is defined over a finite étale extension of $\mathcal{O}_F[\frac{1}{N}]$.
For almost all places $v$, $z_v$ is the inflation of an unramified cocycle, and since $H^1(\mathcal{O}_{F_v}, \mathbf{G}^*)$ is trivial we get an isomorphism $\Xi'_v$ between the models of $\mathbf{G}^*$ and $\mathbf{G}$ over $\mathcal{O}_{F_v}$, well-defined up to the image of $\mathbf{G}^*(\mathcal{O}_{F_v})$ in $\mathbf{G}^*_{\mathrm{ad}}(\mathcal{O}_{F_v})$.
Finally, at almost all places $v$ the Whittaker datum $\mathfrak{w}_v$ is compatible with the choice of hyperspecial maximal compact subgroup $\mathbf{G}^*(\mathcal{O}_{F_v})$ of $\mathbf{G}^*(F_v)$ (in the sense of \cite{CasSha}), and this compatibility can be pulled back to $\mathbf{G}$ via $(\Xi'_v)^{-1}$.
At the finite set of places not in $S$ where $v$ is Archimedean or $\mathbf{G}$ is ramified or $z_v$ is ramified, there exists $g \in \mathbf{G}^*(\overline{F_v})$ such that $\Xi_v' := \mathrm{Ad}(g) \circ \Xi_v$ is defined over $F_v$, and $\Xi_v'$ is well-defined up to $\mathrm{G}^*_{\mathrm{ad}}(F_v)$.
If the image of $z_v \in H^1(F_v, \mathbf{G}^*)$ is trivial we can further impose that $z_v(\sigma) = g \sigma(g)^{-1}$ for all $\sigma \in \mathrm{Gal}_{F_v}$, and $\Xi_v$ is then well-defined up to the image of $\mathbf{G}^*(F_v)$ in $\mathbf{G}^*_{\mathrm{ad}}(F_v)$.

If $\mathbf{G}^* = \mathbf{SO}_{2n}^{\alpha}$, let $\theta$ be an outer automorphism of $\mathbf{G}$ that is realised by an element having determinant $-1$ in the corresponding orthogonal group.
Choose an automorphism $\theta^*$ of $\mathbf{G}^*$ similarly.
Denote by $\theta_v$ and $\theta_v^*$ the induced automorphisms of $\mathbf{G}_{F_v}$ and $\mathbf{G}^*_{F_v}$.
At all places $v \not\in S$, the automorphisms $\theta_v$ and $\Xi'_v \circ \theta^*_v \circ (\Xi'_v)^{-1}$ of $\mathbf{G}$ are in the same orbit under $\mathbf{G}(F_v)$.

Let $\Psi_{\mathrm{disc}}^{S-\mathrm{AJ}}(\mathbf{G}^*) \subset \Psi_{\mathrm{disc}}(\mathbf{G}^*)$ be the set of $\dot{\psi}$ such that for any $v \in S$, the localisation $\dot{\psi}_v$ is an Adams-Johnson parameter as in section \ref{sec:AJ}.
Let $\dot{\psi} \in \Psi_{\mathrm{disc}}^{S-\mathrm{AJ}}(\mathbf{G}^*)$,
To formulate the multiplicity formula in the next section, we need to define a packet $\Pi_{\dot{\psi}}(\mathbf{G})$ of admissible representations of $\mathbf{G}(\A_F)$, along with a mapping $\Pi_{\dot{\psi}}(\mathbf{G}) \rightarrow (\mathcal{S}_{\psi})^{\vee}$, using the local Arthur packets at our disposal (Propositions \ref{prop:realAJArthur} and \ref{prop:trivit}).
This was done by Kaletha \cite[Proposition 4.2]{Kalglob} in greater generality.
However, as our current setting is slightly different from Kaletha's (we did not lift the cocycles to $\mathbf{G}_{\mathrm{sc}}^*$ as Kaletha did) and simpler than the general case since the Hasse principle holds for all the groups considered in this paper, it is worth recalling how Kaletha's construction specialises to the cases at hand.

Let $\Pi_{\dot{\psi}}(\mathbf{G})$ be the set of restricted tensor products $\pi = \otimes'_v \pi_v$ where for all places $v$ of $F$, $\pi_v \in \Pi_{\dot{\psi}_v}(\mathbf{G}(F_v))$ if $v \not\in S$ and $\pi_v \in \Pi_{\dot{\psi}_v}^{\mathrm{AJ}}(\mathbf{G})$ if $v \in S$, and for almost all places $v$ the character $\langle \cdot, \pi_v \rangle$ is trivial.
For any such $\pi$ let us define a character $\langle \cdot, \pi \rangle$ of $\mathcal{S}_{\dot{\psi}}$.
Recall that $\mathcal{S}_{\dot{\psi}} = C_{\dot{\psi}} / Z(\widehat{\mathbf{G}})^{\mathrm{Gal}_F}$ and that $C_{\dot{\psi}}^+$ denotes the preimage of $C_{\dot{\psi}}$ in $\widehat{\overline{\mathbf{G}}}$.
Thus it is equivalent to define a character on $C_{\dot{\psi}}^+$ which is trivial on $Z(\widehat{\overline{\mathbf{G}}})^+$.
Note that in general each $\langle \cdot, \pi_v \rangle$ is a character of $\mathcal{S}_{\dot{\psi}_v}^+$ in the ``representation of finite groups'' sense, but as we observed in Remark \ref{rema:trivit}, in our setting it is simply a morphism $\mathcal{S}_{\dot{\psi}_v}^+ \rightarrow \C^{\times}$.
Its restriction $\chi_v$ to $\pi_0(Z(\widehat{\overline{\mathbf{G}}})^{+,v})$ is the image of $z$ via \cite[Corollary 5.4]{Kal}.
Thanks to the natural morphisms $\mathcal{S}_{\dot{\psi}}^+ \rightarrow \mathcal{S}_{\dot{\psi}_v}^+$ reviewed in section \ref{sec:localisationpsi} we can define $\langle \cdot, \pi \rangle = \prod_v \langle \cdot, \pi_v \rangle|_{\mathcal{S}_{\dot{\psi}}^+}$.
By \cite[Corollary 3.45]{Kalglob} we have $\prod_v \chi_v|_{\pi_0\left(Z(\widehat{\overline{\mathbf{G}}})^+\right)} = 1$ and so $\langle \cdot, \pi \rangle$ is trivial on $\pi_0\left(Z(\widehat{\overline{\mathbf{G}}})^+\right)$ and descends to a character of $\mathcal{S}_{\dot{\psi}}$.

For $s \in \mathcal{S}_{\dot{\psi}}$ let $\mathfrak{e} = (\mathbf{H}, \mathcal{H}, s, \xi)$ be the induced elliptic endoscopic datum.
Sensu stricto it depends on the choice of a lift of $s$ in $Z(\widehat{\mathbf{G}}) C_{\dot{\psi}}$, but multiplication by an element of $Z(\widehat{\mathbf{G}})$ does not change the isomorphism class of $\mathfrak{e}$.
Fix a compatible L-embedding ${}^L \xi : {}^L \mathbf{H} \rightarrow {}^L \mathbf{G}$ as in section \ref{sec:ellenddat}, yielding $\dot{\psi}' \in \Psi_{\mathbf{G}-\mathrm{disc}}(\mathbf{H})$ such that $\dot{\psi} = {}^L \xi \circ \dot{\psi}'$ as reviewed in section \ref{sec:subsALparam}.
For any choice of $\dot{s} \in \mathcal{S}_{\dot{\psi}}^+$ above $s$ we get local refined endoscopic data $\dot{\mathfrak{e}}_v$ above $\mathfrak{e}_v$, and thus absolute transfer factors $\Delta'[\dot{\mathfrak{e}}_v, {}^L \xi, \mathbf{G}, \Xi_v, z_v]$.
By \cite[Proposition 4.1]{Kalglob} the product of these local transfer factors is the \emph{canonical} adélic absolute transfer factor $\Delta'[\mathfrak{e}, {}^L \xi, \mathbf{G}, \Xi]$.
Recall that this transfer factor does not depend on the choices of $z$ and $\mathfrak{w}$, nor on the choice of $\dot{s}$ above $s$.
Let $f(g)dg = \prod_v f_v(g_v)dg_v$ be a smooth, compactly supported distribution on $\mathbf{G}(\A_F)$ such that for any place $v$, the orbital integrals of $f_v(g_v)dg_v$ are $\theta_v$-invariant.
Let $f'(h)dh = \prod_v f'(h_v)dh_v$ be a transfer of $f(g)dg$ to $\mathbf{H} = \mathbf{H}_1 \times \mathbf{H}_2$.
Then the stable orbital integrals of $f'(h)dh$ are invariant under $\mathrm{Out}(\mathbf{H}_{1,F_v}) \times \mathrm{Out}(\mathbf{H}_{2,F_v})$ for all places $v$, and we have the endoscopic character relation
\begin{equation} \label{eqn:globalcharrel} \sum_{\pi \in \Pi_{\dot{\psi}}(\mathbf{G})} \langle s s_{\dot{\psi}}, \pi \rangle \mathrm{tr}\left(\pi(f(g)dg)\right) = \sum_{\pi' \in \Pi_{\dot{\psi}'}(\mathbf{H})} \langle s_{\dot{\psi}'}, \pi' \rangle \mathrm{tr}\left(\pi'(f'(h)dh)\right) \end{equation}
where the packet $\Pi_{\dot{\psi}'}(\mathbf{H})$ on the right hand side is the one defined by Arthur in \cite[(1.5.3)]{Arthur}.
This relation is an immediate consequence of Propositions \ref{prop:realAJArthur} and \ref{prop:trivit} and the well-known product formula $\prod_v e(\mathbf{G}_{F_v}) = 1$ for the Kottwitz signs.
Note that $e(\mathbf{G}_{F_v}) = 1$ for any finite place $v$ because $\mathbf{G}_{F_v}$ is quasi-split.

\section{Arthur's multiplicity formula for certain inner forms}

We continue with the definitions and notations of section \ref{sec:adelicApack}.
For any Archimedean place $v$ of $F$, fix a maximal compact subgroup $K_v$ of $\mathbf{G}(F_v)$ and denote by $\mathfrak{g}_v$ the reductive complex Lie algebra $\C \otimes_{\R} \mathrm{Lie} \left(\mathrm{Res}_{F_v / \R} (\mathbf{G}_{F_v}) \right)$.
Let $\mathcal{A}(\mathbf{G})$ denote the space of automorphic forms for $\mathbf{G}$ in the sense of \cite{BoJaCorv}, with respect to $\prod_{v | \infty} K_v$.
Let $\mathcal{A}^2(\mathbf{G})$ denote the subspace of square-integrable automorphic forms.
It decomposes as a countable orthogonal direct sum of admissible irreducible unitary representations $\pi = \otimes'_v \pi_v$ of $\mathbf{G}(\A_F)$:
$$ \mathcal{A}^2(\mathbf{G}) \simeq \bigoplus_{\pi} \pi^{\oplus m(\pi)}. $$
Note that for $v$ Archimedean $\pi_v$ is really an irreducible unitary $(\mathfrak{g}_v, K_v)$-module, and our use of the term ``representation'' is somewhat abusive.
For any Archimedean place $v$ of $F$ the action of the center $\mathcal{Z}(\mathfrak{g}_v)$ of the enveloping algebra of $\mathfrak{g}_v$ on $\mathcal{A}^2(\mathbf{G})$ is semi-simple, since $\mathcal{Z}(\mathfrak{g}_v)$ acts on $\pi_v$ via the infinitesimal character $\nu_{\pi_v} : \mathcal{Z}(\mathfrak{g}_v) \rightarrow \C$.
Recall that via the Harish-Chandra isomorphism such characters are in bijection with semi-simple conjugacy classes in the complex Lie algebra $\widehat{\mathfrak{g}_v}$ dual to $\mathfrak{g}_v$.
These correspond naturally to semi-simple conjugacy classes in the Lie algebra of the complex Lie group $\widehat{\mathbf{G}}$ if $v$ is real, and pairs of such conjugacy classes if $v$ is complex.
For an infinitesimal character $\nu = (\nu_v)_{v | \infty}$, denote by $\mathcal{A}^2(\mathbf{G}, \nu)$ the subspace of $\mathcal{A}^2(\mathbf{G})$ on which $\mathcal{Z}(\mathfrak{g}_v)$ acts by $\nu_v$ for all Archimedean $v$.
If $\mathbf{G}$ is even orthogonal, we will also consider the $\prod_{v | \infty}\mathrm{Out}(\mathbf{G}_{F_v})$-orbit $\tilde{\nu}$ of $\nu$ instead, and $\mathcal{A}^2(\mathbf{G}, \tilde{\nu})$ will denote the direct sum of the subspaces $\mathcal{A}^2(\mathbf{G}, \nu')$ for $\nu'$ in the orbit of $\nu$.

\begin{theo} \label{theo:main}
Let $f(g)dg = \prod_v f_v(g_v)dg_v$ be a smooth, compactly supported distribution on $\mathbf{G}(\A_F)$ such that for any Archimedean place $v$ of $F$, the function $f_v$ is bi-$K_v$-finite.
If $\mathbf{G}$ is even orthogonal, assume also that for any place $v$ of $F$ the orbital integrals of $f_v(g_v)dg_v$ are $\theta_v$-invariant.
Let $\nu = (\nu_v)_{v | \infty}$ be an infinitesimal character as above, and denote by $\tilde{\nu}$ its $\prod_{v | \infty}\mathrm{Out}(\mathbf{G}_{F_v})$-orbit, which only consists of $\nu$ if $\mathbf{G}$ is symplectic or odd orthogonal.
Assume that for all $v$ in the non-empty set $S$, the infinitesimal character $\nu_v$ is algebraic regular.
Denote by $\Psi_{\mathrm{disc}}(\mathbf{G}, \tilde{\nu})$ the subset of $\Psi_{\mathrm{disc}}(\mathbf{G})$ consisting of those $\dot{\psi}$ such that for all Archimedean $v$, the infinitesimal character of $\dot{\psi}_v$ belongs to $\tilde{\nu}_v$.

Then for all $\dot{\psi} \in \Psi_{\mathrm{disc}}(\mathbf{G}, \tilde{\nu})$ and $v \in S$, $\dot{\psi}_v$ is an Adams-Johnson parameter in the sense of section \ref{sec:AJ}, and
\begin{equation} \label{eqn:main}
\mathrm{tr}\left( f(g)dg\,\middle|\,\mathcal{A}^2(\mathbf{G}, \tilde{\nu}) \right) = \sum_{\substack{\dot{\psi} \in \Psi_{\mathrm{disc}}(\mathbf{G}, \tilde{\nu}) \\ \text{up to } \widehat{\mathbf{G}}-\mathrm{conj}}} \  \sum_{\substack{\pi \in \Pi_{\dot{\psi}}(\mathbf{G}) \\ \text{s.t.\ } \langle \cdot, \pi \rangle = \epsilon_{\psi}}} \mathrm{tr}(\pi(f(g)dg))
\end{equation}
where $\epsilon_{\psi}$ is the character of $\mathcal{S}_{\psi}$ defined by Arthur in \cite[(1.5.6)]{Arthur}.
\end{theo}
\begin{proof}
We use the stabilisation of the trace formula for $\mathbf{G}$ (\cite[(3.2.3)]{Arthur}, proved in \cite{ArthurSTF1}, \cite{ArthurSTF2}, \cite{ArthurSTF3}.
This technique is hardly new: using the stabilisation of the trace formula for this purpose was initiated in \cite{LabLan} for the group $\mathbf{SL}_2$ and its inner forms.
We benefited from reading Kottwitz' conjectural but very general stabilisation of the cuspidal tempered part of the spectral side of the trace formula \cite{KottSTFcusp}, and of course Arthur's book \cite{Arthur}.

We refer the reader to \cite[§4]{ArthurITF} and \cite[§3.1]{Arthur} for details concerning the spectral side of the trace formula.
For each Archimedean place $v$ of $F$, choose a Cartan subalgebra $\widehat{\mathfrak{h}_v}$ of the complex Lie algebra $\widehat{\mathfrak{g}_v}$ dual to $\mathfrak{g}_v$.
Note that $\widehat{\mathfrak{h}_v}$ has a natural integral structure since it is the Lie algebra of a maximal torus in the dual group of $\mathrm{Res}_{F_v / \R}(\mathbf{G}_{F_v})$.
In particular it has a real structure, and we can choose a Euclidean norm on $\mathrm{Re}(\widehat{\mathfrak{h}_v})$ invariant under the Weyl group, yielding an invariant norm $|| \cdot ||_v$ on $\widehat{\mathfrak{h}_v}$.
Clearly we can also take this norm to be invariant under outer automorphisms.
The discrete part of the spectral side of the trace formula for $\mathbf{G}$ is \emph{formally} the sum over $t \in \R_{\geq 0}$ of $I_{\mathrm{disc}, t}^{\mathbf{G}}(f(g)dg)$, and following Arthur we considers these terms individually.
Fix a finite set $T$ of irreducible continuous representations of $\prod_{v | \infty} K_v$ and a compact open subgroup $U$ of $\mathbf{G}(\A_{F, f})$.
By \cite[Lemma 4.1]{ArthurITF} the linear form $I_{\mathrm{disc}, t}^{\mathbf{G}}$ on the space of smooth compactly supported distributions which are bi-$U$-invariant and bi-$\left(\prod_{v | \infty} K_v\right)$-finite with type in $T$ is a linear combination of traces in irreducible unitary admissible representations $\pi = \otimes'_v \pi_v$ of $\mathbf{G}(\A_F)$ such that $\sum_{v | \infty} ||\mathrm{Im}(\nu_{\pi_v})||_v = t$.
We stress that in general $I_{\mathrm{disc}, t}^{\mathbf{G}}(f(g)dg)$ is \emph{not} simply the sum of $m(\pi) \mathrm{tr}(\pi(f(g)dg))$ over the set of discrete automorphic representations $\pi$ for $\mathbf{G}$ such that $\sum_{v | \infty} ||\mathrm{Im}(\nu_{\pi_v})||_v = t$, as there are additional terms indexed by proper Levi subgroups of $\mathbf{G}$ containing a given minimal one: see \cite[3.1.1]{Arthur}.
These additional representations of $\mathbf{G}(\A_F)$ contributing to $I_{\mathrm{disc}, t}^{\mathbf{G}}(f(g)dg)$, however, have non-regular infinitesimal character at all Archimedean places of $F$.

The stabilisation of the discrete part of the trace formula is the equality
\begin{equation} \label{eqn:STF}
 I_{\mathrm{disc}, t}^{\mathbf{G}}(f(g)dg) = \sum_{\mathfrak{e}} \iota(\mathfrak{e}) S_{\mathrm{disc}, t}^{\mathbf{H}}(f'(h)dh)
\end{equation}
where the sum is over the isomorphism classes of elliptic endoscopic data $\mathfrak{e} = (\mathbf{H}, \mathcal{H}, s, \xi)$ for $\mathbf{G}$, and $\iota(\mathfrak{e}) = \tau(\mathbf{G}) \tau(\mathbf{H})^{-1} |\mathrm{Out}(\mathfrak{e})|^{-1}$.
For any such isomorphism class a representative $\mathfrak{e}$ is chosen, along with a compatible L-embedding ${}^L \xi : {}^L \mathbf{H} \rightarrow {}^L \mathbf{G}$.
As explained before this yields canonical adélic transfer factors, and the smooth, compactly supported distribution $f'(h)dh$ on the right hand side of \ref{eqn:STF} can be any transfer of $f(g)dg$.
It is possible to choose $\prod_{v | \infty } f'_v(h_v)$ to be bi-finite for any given maximal compact subgroup of $\mathbf{H}(\R \otimes_{\Q} F)$.
The right hand side makes sense because Arthur has shown that the linear forms $f'(h)dh \mapsto S_{\mathrm{disc}, t}^{\mathbf{H}}(f'(h)dh)$, defined inductively by the equality \ref{eqn:STF} with $\mathbf{G}$ replaced by $\mathbf{H}$, are stable.
Note that there is only a finite number of isomorphism classes of elliptic endoscopic data unramified outside a finite set of places of $F$ by \cite[Lemme 8.12]{Langlandsdebuts}, and combined with \cite[Proposition 7.5]{KottEllSing} this implies that there are only finitely many non-zero terms in the sum on the right hand side of \ref{eqn:STF}.

Arthur gives a spectral expansion for the terms $S_{\mathrm{disc}, t}^{\mathbf{H}}(f'(h)dh)$, under the assumption that the stable orbital integrals of $f'(h)dh$ are invariant under $\mathrm{Out}(\mathbf{H}_{1,F_v}) \times \mathrm{Out}(\mathbf{H}_{2,F_v})$ at all places $v$ of $F$.
By \cite[Lemma 3.3.1, Proposition 3.4.1 and Theorem 4.1.2]{Arthur},
$$ S_{\mathrm{disc}, t}^{\mathbf{H}}(f'(h)dh) = \sum_{\substack{\psi' \in \widetilde{\Psi}(\mathbf{H}) \\ \text{s.t. } t(\psi') = t}} m_{\psi'} |\mathcal{S}_{\psi'}|^{-1} \sigma(\overline{S}^0_{\psi'}) \epsilon_{\psi'}(s_{\psi'}) \Lambda_{\psi'}(f'(h)dh). $$
Here $\Lambda_{\psi'} = \prod_v \Lambda_{\psi'_v}$ is the stable linear form associated to $\psi' = (\psi'_1, \psi'_2)$ (see \ref{eqn:defstable}), and $\widetilde{\Psi}(\mathbf{H}) = \widetilde{\Psi}(\mathbf{H}_1)  \times \widetilde{\Psi}(\mathbf{H}_2)$.
In general the set $\widetilde{\Psi}(\mathbf{H}_i)$ properly contains the set of \emph{discrete} parameters $\widetilde{\Psi}_{\mathrm{disc}}(\mathbf{H}_i)$ reviewed in section \ref{sec:subsALparam}.
These substitutes for non-discrete Arthur-Langlands parameters are also defined as formal unordered direct sums $\psi'_i = \boxplus_j \pi_{i,j} [d_{i,j}]^{\boxplus l_{i,j}}$, where $\pi_{i,j}$ is a unitary automorphic cuspidal representation of $\mathbf{GL}_{n_{i,j}} / F$, $d_{i,j} \geq 1$, $l_{i,j} \geq 1$, the pairs $(\pi_{i,j}, d_{i,j})$ are distinct, and if $\pi_{i,j}$ is not self-dual then there is a unique index $k$ such that $\pi_{i,k}^{\vee} \simeq \pi_{i,j}$ and $d_{i,k} = d_{i,j}$, and it satisfies $l_{i,k} = l_{i,j}$.
As in the discrete case one can consider at each place $v$ of $F$ the localisation $\psi'_{i,v}$, an Arthur-Langlands parameter $\mathrm{WD}_{F_v} \times \mathrm{SL}_2(\C) \rightarrow {}^L \mathbf{H}_i$ which is well-defined up to the action of $\mathrm{Aut}({}^L \mathbf{H}_i)$.
The number $t(\psi')$ is defined as the norm of the imaginary part of the infinitesimal character of ${}^L \xi (\psi')$.
We do not recall the definition of $\sigma(\overline{S}^0_{\psi'})$ because eventually we will not need it.

By \cite[Lemma 24]{Mezo} or \cite[Corollaire 2.8]{WalTTF1}, the transfer correspondence $f_v(g_v)dg_v \leftrightarrow f'_v(h_v)dh_v$ at an Archimedean place $v$ of $F$ is compatible with infinitesimal characters in the following sense.
The L-embedding ${}^L \xi$ yields an algebraic mapping from the complex algebraic variety of infinitesimal characters for $\mathbf{H}_{F_v}$ to the analogue variety for $\mathbf{G}_{F_v}$.
For the cases considered in this paper this is particularly simple since it is not necessary to take a z-extension of $\mathbf{H}$, and ${}^L \xi$ is defined between L-groups formed with Galois groups rather than Weil groups.
Dually we get morphisms of $\C$-algebras $({}^L \xi)^*_v : \mathcal{Z}(\mathfrak{g}_v) \rightarrow \mathcal{Z}(\mathfrak{h}_v)$.
The algebra $\mathcal{Z}(\mathfrak{g}_v)$ acts on the space of bi-$K_v$-finite smooth compactly supported distributions on $\mathbf{G}(F_v)$, and preserves any subspace defined by restricting the $K_v$-types.
The aforementioned compatibly is that if $f'_v(h_v)dh_v$ is a transfer of $f_v(g_v)dg_v$, then for any $X \in \mathcal{Z}(\mathfrak{g}_v)$, $\left(({}^L \xi)^*_v(X) \cdot f'_v\right)(h_v) dh_v$ is a transfer of $(X \cdot f_v)(g_v)dg_v$.
Therefore, for any infinitesimal character $\nu = (\nu_v)_{v | \infty}$ for $\mathbf{G}$ such that $\sum_{v | \infty} || \mathrm{Im}(\nu_v) ||_v = t$, the stabilisation of the trace formula \ref{eqn:STF} can be refined as
$$ I_{\mathrm{disc}, \tilde{\nu}}^{\mathbf{G}}(f(g)dg) = \sum_{\mathfrak{e}} \iota(\mathfrak{e}) \sum_{\psi' \in \widetilde{\Psi}(\mathbf{H}, \tilde{\nu})} m_{\psi'} |\mathcal{S}_{\psi'}|^{-1} \sigma(\overline{S}^0_{\psi'}) \epsilon_{\psi'}(s_{\psi'}) \Lambda_{\psi'}(f'(h)dh) $$
where $\widetilde{\Psi}(\mathbf{H}, \tilde{\nu})$ is the subset of $\widetilde{\Psi}(\mathbf{H})$ consisting of those $\psi'$ such that for any Archimedean place $v$, the infinitesimal character of ${}^L \xi (\psi'_v)$ belongs to the $\mathrm{Out}(\mathbf{G}_{F_v})$-orbit of $\nu_v$, and $I_{\mathrm{disc}, \tilde{\nu}}^{\mathbf{G}}$ is obtained from the linear combination $I_{\mathrm{disc}, t}^{\mathbf{G}}$ of traces in irreducible unitary admissible representations of $\mathbf{G}(\A_F)$ by keeping only those representations whose infinitesimal character belongs to $\tilde{\nu}$.
Of course this compatibility was already used in the proof of the trace formula and in \cite{Arthur}.

From now on we assume that $\nu_v$ is algebraic regular for all $v \in S$.
We claim that for $v \in S$ the Arthur-Langlands parameter ${}^L \xi (\psi'_v)$, whose $\widehat{\mathbf{G}}$-conjugacy class is only defined up to outer conjugacy if $\mathbf{G}$ is even orthogonal, is an Adams-Johnson parameter in the sense of section \ref{sec:AJ}.
We only have to check that the restriction of $\psi'_v$ to $W_{\overline{F_v}} \times \{1\} \subset W_{F_v} \times \mathrm{SL}_2(\C)$ is bounded.
For any $i$ and $j$ the infinitesimal character of the Langlands parameter $W_{F_v} \rightarrow \mathrm{GL}_{n_{i,j}}(\C)$ associated to $(\pi_{i,j})_v$ or $(\pi_{i,j})_v |\det|^{1/2}$ is algebraic, and by Clozel's purity lemma \cite[Lemme 4.9]{ClozelPurity} this implies that $(\pi_{i,j})_v$ is tempered, i.e.\ its Langlands parameter is bounded.

Since $S \neq \emptyset$, this implies that ${}^L \xi (\psi')$ is discrete, and in particular $\psi'$ itself is discrete, which implies that $\sigma(\overline{S}^0_{\psi'}) = 1$.
Moreover, as we recalled above, the contributions from proper Levi subgroups of $\mathbf{G}$ to the discrete part of the spectral side of the trace formula for $\mathbf{G}$ have non-regular infinitesimal character at all Archimedean places of $F$, thus the stabilisation of the trace formula becomes the more concrete equality
$$ \mathrm{tr}\left(f(g)dg \,\middle|\, \mathcal{A}^2(\mathbf{G}, \tilde{\nu})\right) = \sum_{\mathfrak{e}} \iota(\mathfrak{e}) \sum_{\psi' \in \widetilde{\Psi}_{\mathbf{G}-\mathrm{disc}}(\mathbf{H}, \tilde{\nu})} m_{\psi'} |\mathcal{S}_{\psi'}|^{-1} \epsilon_{\psi'}(s_{\psi'}) \Lambda_{\psi'}(f'(h)dh) $$
where $\widetilde{\Psi}_{\mathbf{G}-\mathrm{disc}}(\mathbf{H}, \tilde{\nu}) = \widetilde{\Psi}(\mathbf{H}, \tilde{\nu}) \cap \widetilde{\Psi}_{\mathbf{G}-\mathrm{disc}}(\mathbf{H})$.
The right hand side can also be written as
$$ \sum_{\mathfrak{e}} \iota(\mathfrak{e}) \sum_{\substack{\dot{\psi}' \in \Psi_{\mathbf{G}-\mathrm{disc}}(\mathbf{H}, \tilde{\nu}) \\ \text{up to } \widehat{\mathbf{H}}-\mathrm{conj.}}} |\mathcal{S}_{\psi'}|^{-1} \epsilon_{\psi'}(s_{\psi'}) \Lambda_{\psi'}(f'(h)dh). $$
We now want to express this using discrete parameters for $\mathbf{G}$, using Proposition \ref{prop:endobij}.
Let us write $\dot{\psi} = {}^L \xi \circ \dot{\psi}'$, and recall that $s$ defines an element of $\mathcal{S}_{\dot{\psi}}$.
By \cite[Lemma 4.4.1]{Arthur} we have $ \epsilon_{\psi'}(s_{\psi'}) = \epsilon_{\psi}(ss_{\psi}) $, and by equation \ref{eqn:globalcharrel} we have
$$ \Lambda_{\psi'}(f'(h)dh) = \sum_{\pi \in \Pi_{\psi}(\mathbf{G})} \langle ss_{\psi}, \pi \rangle \mathrm{tr}(\pi(f(g)dg))  $$
The stabiliser of the $\widehat{\mathbf{H}}$-conjugacy class of $\dot{\psi}'$ under the action of $\mathrm{Out}(\mathfrak{e})$ is naturally identified with
$$ \left(\mathrm{Aut}(\mathfrak{e}) \cap C_{\dot{\psi}} \right) / \xi(C_{\dot{\psi}'}) $$
and $\left(\mathrm{Aut}(\mathfrak{e}) \cap C_{\dot{\psi}}\right) / Z(\widehat{\mathbf{G}})^{\mathrm{Gal}_F} = \mathrm{Cent}(s, \mathcal{S}_{\dot{\psi}})$ is simply equal to $\mathcal{S}_{\dot{\psi}}$ because $\mathcal{S}_{\dot{\psi}}$ is abelian.
Thus there are
\begin{equation} \label{eqn:numorbit}
\frac{ |\mathrm{Out}(\mathfrak{e})| \times |\mathcal{S}_{\dot{\psi}'}| \times |\xi(Z(\widehat{\mathbf{H}})^{\mathrm{Gal}_F}) / Z(\widehat{\mathbf{G}})^{\mathrm{Gal}_F}|}{|\mathcal{S}_{\dot{\psi}}|}
\end{equation} 
$\widehat{\mathbf{H}}$-conjugacy classes in $\Psi_{\mathbf{G}-\mathrm{disc}}(\mathbf{H})$ that map to the $\widehat{\mathbf{G}}$-conjugacy class of $(\dot{\psi}, sZ(\widehat{\mathbf{G}}))$.
The groups $\mathbf{G}$ and $\mathbf{H}$ satisfy the Hasse principle, therefore $\tau(\mathbf{G}) = |\pi_0(Z(\widehat{\mathbf{G}})^{\mathrm{Gal}_F})|$ and $\tau(\mathbf{H}) = |\pi_0(Z(\widehat{\mathbf{H}})^{\mathrm{Gal}_F})|$, and finally we can conclude that \ref{eqn:numorbit} equals $ |\mathcal{S}_{\dot{\psi}'}| \iota(\mathfrak{e})^{-1} |\mathcal{S}_{\dot{\psi}}|^{-1}$.
By grouping the terms corresponding to $\mathrm{Out}(\mathfrak{e})$-orbits in the sum, we get
\begin{align*} \mathrm{tr}\left(f(g)dg \,\middle|\, \mathcal{A}^2(\mathbf{G}, \tilde{\nu})\right) & = \sum_{\substack{\dot{\psi} \in \Psi_{\mathrm{disc}}(\mathbf{G}, \tilde{\nu}) \\ \text{up to } \widehat{\mathbf{G}}-\mathrm{conj}} } \, \sum_{s \in \mathcal{S}_{\psi}} \, \sum_{\pi \in \Pi_{\psi}(\mathbf{G})} |\mathcal{S}_{\psi}|^{-1} \epsilon_{\psi}(s s_{\psi}) \langle s s_{\psi}, \pi \rangle \mathrm{tr}(\pi(f(g)dg)) \\
&= \sum_{\substack{\dot{\psi} \in \Psi_{\mathrm{disc}}(\mathbf{G}, \tilde{\nu}) \\ \text{up to } \widehat{\mathbf{G}}-\mathrm{conj}} } \  \sum_{\substack{\pi \in \Pi_{\psi}(\mathbf{G})\\ \mathrm{s.t. } \langle \cdot, \pi \rangle = \epsilon_{\psi}}} \mathrm{tr}(\pi(f(g)dg))
\end{align*}
\end{proof}

\begin{rema}
\begin{enumerate}
\item
We formulated Theorem \ref{theo:main} using traces of distributions on the discrete automorphic spectrum, but of course this is equivalent to a formulation similar to that of Theorem 1.5.2 in \cite{Arthur}.
\item
If $\mathbf{G}$ is such that $\mathbf{G}(\R \otimes_{\Q} F)$ is compact, in which case one may take for $S$ the set of all real places of $F$, Theorem \ref{theo:main} describes the full automorphic spectrum of $\mathbf{G}$.
\end{enumerate}
\end{rema}

\newpage

\bibliographystyle{amsalpha}
\bibliography{cpctmult}

\end{document}